\documentclass[11pt]{article}
\usepackage[a4paper,margin=28mm]{geometry}
\usepackage[T1]{fontenc}
\usepackage{lmodern,amsmath,amssymb,amsthm,mathtools,microtype,booktabs}
\usepackage{graphicx}
\usepackage[dvipsnames]{xcolor}
\usepackage[colorlinks=true,linkcolor=MidnightBlue,citecolor=MidnightBlue,urlcolor=MidnightBlue]{hyperref}
\usepackage{enumitem}
\setlist[itemize]{leftmargin=2em,topsep=3pt,itemsep=2pt}
\allowdisplaybreaks[2]
\numberwithin{equation}{section}
\newtheorem{theorem}{Theorem}[section]
\newtheorem{proposition}[theorem]{Proposition}
\newtheorem{lemma}[theorem]{Lemma}
\newtheorem{corollary}[theorem]{Corollary}
\newtheorem{conjecture}[theorem]{Conjecture}
\theoremstyle{remark}
\newtheorem{remark}[theorem]{Remark}
\newcommand{\R}{\mathbb R}
\newcommand{\E}{\mathbb E}
\newcommand{\Var}{\operatorname{Var}}

\newcommand{\IG}{I_{\mathrm G}}
\newcommand{\dd}{\,d}
\newcommand{\one}{\mathbf 1}
\newcommand{\cF}{\mathcal F}
\newcommand{\cE}{\mathcal E}
\newcommand{\cQ}{\mathcal Q}

\newcommand{\symdiff}{\mathbin{\triangle}}
\newcommand{\divg}{\operatorname{div}_{\gamma}}
\newcommand{\cG}{\mathcal G}
\newcommand{\clip}{\operatorname{clip}}
\title{Dimension-Free Isoperimetry Near Gaussian Contact Points}
\author{Lei Yu\thanks{School of Statistics and Data Science, LPMC, KLMDASR, and LEBPS,
Nankai University, Tianjin 300071, China. E-mail: leiyu@nankai.edu.cn.}}
\date{}
\hypersetup{
 pdftitle={Dimension-Free Isoperimetry Near Gaussian Contact Points},
 pdfauthor={Lei Yu},
 pdfsubject={Gaussian contact, dimension-free product isoperimetry, and quantitative central intervals},
 pdfkeywords={isoperimetry, product measures, Riemannian manifolds, local optimality, rigidity}
}
\begin{document}
\maketitle
\begin{abstract}
We establish a dimension-free Gaussian-contact theorem for products of weighted Riemannian manifolds. If the one-factor isoperimetric profile is concave and $C^2$, dominates a scaled Gaussian profile, and touches it at an arbitrary interior volume $b$, then the product isoperimetric profile coincides with the one-factor profile on a nontrivial dimension-independent neighborhood of both $b$ and $1-b$. When the one-factor profile is not Gaussian, all minimizers in these neighborhoods are coordinate cylinders. This extends Barthe's Gaussian-comparison theorem (2001) from a contact volume to an interval. For products of intervals and circles, we establish an explicit central half-width of $9/625=0.0144$, resolving Glaudo's dimension-free isoperimetric open problem (2024).
\end{abstract}


\section{Introduction}\label{sec:intro}  
The study of dimension-free isoperimetric profiles of product
measures goes back to the pioneering work of Bobkov
\cite{Bobkov1997Isoperimetric}, Bobkov and Houdr\'e
\cite{BobkovHoudre1997}, and Barthe \cite{Barthe2004},
particularly in the setting of uniform enlargement associated
with the $\ell_\infty$ product metric.
From the geometric viewpoint, however, the Euclidean product
metric is more natural, as it coincides with the Riemannian
product metric; see also \cite{Ros2005,BartheHuou2017}.
While dimension-free isoperimetric inequalities for Euclidean
products have been extensively investigated
\cite{BobkovHoudre1997,Barthe2001,Ros2005,BartheCattiauxRoberto2007},
determining their exact isoperimetric profiles remains
substantially more difficult, even for the uniform measure
on the cube \cite{Glaudo2024,Milman2026}.
The present work first establishes an exact dimension-free isoperimetric identity near any interior contact point between the one-factor profile and a scaled Gaussian profile, and then derives explicit dimension-free central intervals for Cartesian powers of the unit interval and the unit circle.
 
The distinction between fixed-dimensional and dimension-free optimality is particularly evident in the continuous cube (i.e., the Cartesian power of the unit interval). For the unit cube $Q_n=(0,1)^n$, coordinate slabs have relative perimeter one and, as shown by   Glaudo \cite{Glaudo2024}, they are known to minimize perimeter for volumes in a neighborhood of $1/2$ in each fixed dimension. As discussed by Glaudo, this fixed-dimensional result can be traced back to the variational arguments of Acerbi, Fusco, and Morini \cite{AcerbiFuscoMorini2013}. However, whether such a neighborhood can be chosen independently of the dimension was explicitly raised as an open problem in \cite[Open Problem~1.1]{Glaudo2024}. Related questions were also discussed by Brezis \cite[Open Problem~10.1]{Brezis2023} and Brezis--Bruckstein \cite[Remark~2]{BrezisBruckstein2021}. The main difficulty is that fixed-dimensional optimality and local perturbation arguments do not necessarily provide estimates that remain uniform as the dimension grows, particularly in the presence of high-dimensional competitors that cannot be detected by such arguments.

The cube problem naturally leads to a broader question for isoperimetry in product spaces: for which probability measures $\mu$ on Euclidean spaces or general smooth Riemannian manifolds does the isoperimetric optimality of sets depending on a single factor persist throughout a nontrivial interval around half-volume, uniformly over all product dimensions? Barthe's Gaussian-comparison theorem \cite{Barthe2001} identifies an important class of measures for which coordinate half-spaces minimize perimeter at half-volume (or at any Gaussian-profile contact points) in every product dimension. The second-variation analysis of Barthe, Bianchini, and Colesanti \cite{BartheBianchiniColesanti2013} further highlights the distinction between local stability and global isoperimetric optimality, as the former alone does not guarantee the latter. Building on Barthe's Gaussian-profile comparison, we establish that this exact optimality persists on a nontrivial interval around half-volume, with a width independent of the product dimension, and derive an explicit quantitative result for the continuous cube.
More generally, the theorem applies to an arbitrary interior contact point
$b$ and its reflected point $1-b$. 

We now provide the mathematical formulation. Let $(M,g)$ be a connected smooth Riemannian manifold
of dimension $d$, possibly with smooth boundary and
connected interior. Let $\mu$ be a probability measure
on $M$ of the form
\[
d\mu=f\,d\operatorname{vol}_g,
\]
where $f$ is strictly positive and $C^2$ on $M^\circ$.
For each $n\ge1$, equip $M^n$ with the product
Riemannian metric $g^{\oplus n}$ and the product
measure $\mu^{\otimes n}$.
For a Borel set $A\subseteq M^n$, let
$P_{\mu^{\otimes n}}(A)$ denote its relative weighted
De Giorgi perimeter. If $A$ has finite perimeter,
then
\begin{equation}\label{eq:perimeter}
P_{\mu^{\otimes n}}(A)
=\int_{\partial^*A\cap(M^\circ)^n}
f_n\,d\mathcal H_g^{dn-1},
\end{equation}
where $f_n(x)=\prod_{i=1}^n f(x_i)$ and
$\partial^*A$ denotes the reduced boundary.
Otherwise, set $P_{\mu^{\otimes n}}(A)=+\infty$.
In particular, boundary contributions from
$\partial M^n$ are excluded.
The isoperimetric profile is defined by
\begin{equation}\label{eq:Gamma}
I_n(a):=I_{\mu^{\otimes n}}(a)
=\inf\left\{
P_{\mu^{\otimes n}}(A):
A\subseteq M^n\text{ Borel},\
\mu^{\otimes n}(A)=a
\right\},
\qquad a\in[0,1].
\end{equation}
For brevity, write $I:=I_1$. The asymptotic
(or dimension-free) isoperimetric profile is
defined by
\[
I_\infty(a):=\lim_{n\to\infty}I_n(a)
=\inf_{n\ge1}I_n(a),
\]
where the limit exists since $I_{n+1}(a)\le I_n(a)$.

Let $K_{\mathrm{P}}=K_{\mathrm{P}}(\mu)$ and
$K_{\mathrm{LS}}=K_{\mathrm{LS}}(\mu)$ denote the optimal
Poincar\'e and log-Sobolev constants, respectively,
with the spectral-gap normalization
\[
K_{\mathrm{P}}\operatorname{Var}_\mu(h)
\le\int_M|\nabla h|^2\,d\mu,\qquad
K_{\mathrm{LS}}\operatorname{Ent}_\mu(h^2)
\le2\int_M|\nabla h|^2\,d\mu.
\]
Define the Gaussian constant  \cite{Ros2005}  as 
\[
K_{\mathrm{G}}= K_{\mathrm{G}}(\mu):=  \inf_{t\in (0,1)} I(t)/\IG(t).
\]

\subsection{Main results}

We first focus on the one-dimensional Euclidean setting and then extend our result to general manifolds. 
That is, we consider $M=\Omega:=(-L,L)$, $0<L\le\infty$. Let
\begin{equation}\label{eq:measure}
  \dd\mu(x)=f(x)\dd x= e^{-V(x)}\dd x,\qquad
  \int_{\Omega}e^{-V(x)}\dd x =1.
\end{equation}
Here $V\in C^2(\Omega)$ is even and convex. 
Since $\mu$ is even and log-concave, by Bobkov's one-dimensional isoperimetric theorem \cite{Bobkov1996}, the profile for $\mu$ is given by
\begin{equation}\label{eq:I}
  I(a)=f(F^{-1}(a))\quad(0<a<1),\qquad I(0)=I(1)=0.
\end{equation}
Here $F$ is  the
CDF of $\mu$.
The prescribed endpoint values may differ from one-sided limits, as for the
uniform interval.  
For the positive $C^2$ density in \eqref{eq:measure}, 
convexity of $V$ is in fact equivalent
to concavity of $I$. 
Let $\phi$ and $\Phi$ be the standard Gaussian density and CDF respectively,   and $\IG(a)=\phi(\Phi^{-1}(a))$ be its isoperimetric profile.  

\begin{theorem}[Dimension-free isoperimetry near Gaussian contacts]
\label{thm:general}
Let $\mu$ be as in \eqref{eq:measure}, with $V$ even and convex.
Assume there are $\kappa>0$ and a contact point $b\in(0,1)$ such that
\begin{equation}\label{eq:conditions}
 I(a)\ge B(a):=\kappa\IG(a)\qquad(0<a<1),
\end{equation}
and
\begin{equation}\label{eq:contact}
 I(b)=B(b).
\end{equation}
Then there is $\varepsilon_{\mu,b}>0$, depending on the one-factor
measure and the contact point  but not on the product dimension, such that
\begin{equation}\label{eq:generalresult}
 I_n(a)=I_\infty(a)=I(a)
 \qquad\left(n\ge1, 
 \quad |a-b| \le\varepsilon_{\mu,b} \text{ or } |a-(1-b)|\le\varepsilon_{\mu,b} \right).
\end{equation}
If $\mu$ is not Gaussian, every minimizer in this range is a cylinder
over a one-dimensional isoperimetric minimizer; when only endpoint-attached
half-intervals minimize in one dimension, these are coordinate half-spaces.
If $\mu$ is Gaussian, then necessarily $I=\kappa\IG$ and the profile
identity holds at every volume, with arbitrary oriented affine half-spaces
as its minimizers.
\end{theorem}

\begin{corollary}[Central contact at half-volume]\label{cor:central}
Under the hypotheses of Theorem~\ref{thm:general} with
$\kappa= \sqrt{2\pi}\,f(0)$, the contact condition
holds automatically with $b=1/2$. Consequently there exists a
nontrivial dimension-independent interval around half-volume on which
$I_n=I$.\end{corollary}

\begin{remark}
\label{rem:profileconcavity}
For a positive $C^2$ density on an interval, put $J(a)=f(F^{-1}(a))$ and then direct differentiation gives
\begin{equation}\label{eq:concavityequivalence}
 J'(a)=(\log f)'(F^{-1}(a)),\qquad
 J''(a)=\frac{(\log f)''(F^{-1}(a))}{f(F^{-1}(a))}.
\end{equation}
Consequently, concavity of $J$ is equivalent to log-concavity of $f$,
so the formulation of Theorem~\ref{thm:general} in terms of $J$ does
not enlarge its one-dimensional class; see also 
\cite{Bobkov1996} and \cite[Equation~(2.6)]{Milman2026}.  
\end{remark}

Conditions \eqref{eq:conditions} and \eqref{eq:contact} are equivalent to that $K_{\mathrm{G}}$ is attained at $b \in (0,1)$.
An important class of measures satisfying the assumptions of Theorem~\ref{thm:general} is given below.
\begin{corollary}[The square-root class]\label{cor:sqrt}
Under \eqref{eq:measure} with even $V$, suppose $x\mapsto\sqrt{V(x)-V(0)}$ is convex on
$[0,L)$. Then the central conclusion of Corollary~\ref{cor:central} holds. 
\end{corollary}

Barthe \cite{Barthe2001} established the Gaussian comparison for the
square-root class and the exact optimality of coordinate
half-spaces at the half volume (or at other Gaussian contact points). Theorem~\ref{thm:general} gives a common nontrivial
interval not only at the median but near any interior contact point 
between the profile and a scaled Gaussian profile, uniformly over all
product dimensions.
Its proof rests on a dimension-free local-optimality
estimate for BV fiber-volume functions, whose variance coefficient
can approach the sharp value $K_{\mathrm{P}}/(2I(b))$ at the contact point $b$,
and on the quantitative Gaussian rigidity theorem of
Barchiesi, Brancolini, and Julin
\cite{BarchiesiBrancoliniJulin2017}. Neither ingredient alone gives the
claimed interval: the former controls competitors close to coordinate
half-spaces, whereas the latter places all sufficiently competitive sets
inside that neighborhood.

The uniform interval deserves a quantitative statement.
Let $Q_n=(0,1)^n$ denote the unit cube. Its relative
isoperimetric profile is
\begin{equation}\label{eq:cubedef}
I_{Q_n}(a)
:=\inf_{\substack{E\subseteq Q_n\\ |E|=a}}
P(E;Q_n),
\qquad
I_{Q_\infty}(a)
:=\inf_{n\ge1}I_{Q_n}(a).
\end{equation}
Here $P(E;Q_n)$ denotes the relative
De Giorgi perimeter, excluding boundary contributions
from $\partial Q_n$.

As $V\equiv0$ on $(-1/2,1/2)$ belongs to the centered
comparison class, qualitative optimality follows from
Corollary~\ref{cor:central}. A stronger quantitative argument gives the
following interval, with its entire proof in Appendix~\ref{app:circlequant}.

\begin{theorem}[An explicit interval for the uniform cube]\label{thm:cube}
For every $n\ge1$,
\begin{equation}\label{eq:cubemain}
 I_{Q_n}(a)=I_{Q_\infty}(a)=1
 \qquad\text{whenever }|a-\tfrac12|\le\varepsilon_{\mathrm Q},
 \qquad \varepsilon_{\mathrm Q}=\frac9{625}=0.0144.
\end{equation} 
At these volumes, the only minimizers (modulo null sets) are coordinate slabs:
\[ 
E=\{x\in Q_n:x_i\le a\} \text{ or } \{x\in Q_n:x_i\ge 1-a\} \text{ for some } i. 
\] 
\end{theorem}
 
Theorem~\ref{thm:cube} has characterized the limiting isoperimetric profile for cubes at points between $0.5 \pm 0.0144$. However, limiting isoperimetric profile at other points still has not been exactly characterized. 
Glaudo \cite{Glaudo2024} proved that the limiting isoperimetric
profile is strictly greater than the Gaussian lower bound
$\sqrt{2\pi}\,\IG(a)$ for all $a\in(0,1)\setminus\{1/2\}$.
On the other hand, the present author \cite{yu2026sharp}
established the sharp endpoint asymptotics
\[
I_{Q_\infty}(a)
\sim \pi \IG(a)
\qquad \text{as } a\to 0 \text{ or } a\to 1.
\]
These results illustrate the nontrivial structure of the
full limiting isoperimetric profile, whose exact form
remains unknown.
The half-width $9/625$ is an explicit admissible value and is not
claimed to be optimal. A conjectured optimal value is $1/2-1/\pi$, which will be discussed in the next subsection.

The same conclusion can be generalized to product manifold spaces.  Let $(M, g,\mu)$ be a
finite-dimensional connected smooth Riemannian manifold, possibly with
smooth boundary and with connected interior, endowed with a probability
measure having a strictly positive $C^2$ density on its interior. 
Recall that  $I_{n}=I_{\mu^{\otimes n}}$, 
$I=I_1$, and 
$I_{\infty} =\lim_{n\to\infty}I_{n} .$ 
Complements
give $I_{n}(a)=I_{n}(1-a)$ without any reflection symmetry assumption on
$(M, g,\mu)$. We prescribe $I_{n}(0)=I_{n}(1)=0$, allowing a jump from an
endpoint value to its one-sided limit.

\begin{theorem}[Gaussian-contact isoperimetry for weighted manifolds]
\label{thm:manifold}
Let $I=I_{\mu}$ be positive, concave, and $C^2$ on $(0,1)$.
Assume that for some $\kappa>0$ and $b\in(0,1)$,
\begin{equation}\label{eq:manifoldcomparison}
 I(a)\ge\kappa\IG(a)\qquad(0<a<1),
\end{equation}
and
\begin{equation}\label{eq:manifoldcontact}
 I(b)=\kappa\IG(b).
\end{equation}
Then there exists $\varepsilon_{I,b}>0$, depending only on the one-factor
profile and the contact point, such that
\begin{equation}\label{eq:manifoldcentral}
 I_n(a)=I_\infty(a)=I(a) 
 \qquad\left(n\ge1,\quad |a-b| \le\varepsilon_{I,b} \text{ or } |a-(1-b)|\le\varepsilon_{I,b}\right).
\end{equation}
If $I\not\equiv\kappa\IG$, every minimizer in this range (if it exists) is, modulo
null sets, a cylinder over a single-factor minimizing set:
\begin{equation}\label{eq:manifoldcylinder}
 M^{i-1}\times A\times M^{n-i},
 \qquad \mu(A)=a,\quad P_{\mu}(A)=I(a), \text{ for some } i, A.
\end{equation}
If $I\equiv\kappa\IG$, the profile identity holds at every volume,
but no single-factor classification is asserted.
\end{theorem} 

\begin{remark}[Geometric conditions ensuring concavity and existence]\label{rem:CD}
The concavity assumption in Theorem~\ref{thm:manifold}
follows under standard $\mathrm{CD}(0,\infty)$ hypotheses.
The Gaussian comparison, contact condition, and $C^2$ regularity
of the profile must be verified separately.
More precisely, suppose that $(M,g)$ is a $d$-dimensional smooth,
complete, oriented, connected Riemannian manifold ($M=\mathbb{R}$ if $d=1$) and
$
d\mu=e^{-V}\,d\operatorname{vol}_g,
 V\in C^2(M),
$
where $\mu$ is a probability measure satisfying
$
\operatorname{Ric}_g+\operatorname{Hess}_g V \ge0.
$
Then the isoperimetric profile of $(M,g,\mu)$ is
concave on $(0,1)$; see
\cite[Theorem~A.2]{Milman2009}.
Moreover, for every $a\in(0,1)$, an isoperimetric
minimizer of measure $a$ exists. Its boundary is
smooth outside a singular set of Hausdorff
codimension at least $8$; see
\cite[Appendix~A]{Milman2009} and the references therein.
More generally, let $\Omega\subset M$ be a bounded,
locally convex domain with $C^2$ boundary, and let
$\mu$ be the restriction of the above weighted
measure to $\Omega$, normalized to be a probability
measure. If the $\mathrm{CD}(0,\infty)$ condition holds on
$\Omega$, then the relative isoperimetric profile
of $(\Omega,g,\mu)$ is concave on $[0,1]$; see
\cite[Theorem~A.3]{Milman2009}.
Furthermore, for every $a\in(0,1)$, there exists
a relative isoperimetric minimizer of measure $a$.
Its regular boundary is a smooth embedded
hypersurface meeting $\partial\Omega$ orthogonally,
and its singular set has Hausdorff codimension
at least $8$; see
\cite[Theorem~A.3]{Milman2009} and the existence
and regularity results of Morgan and Gr\"uter.
Consequently, under either of these geometric
frameworks, the concavity assumption in
Theorem~\ref{thm:manifold} is automatic, and the
existence of isoperimetric minimizers required for
the equality classification is also guaranteed.
We emphasize, however, that these geometric
conditions are sufficient rather than necessary
for our theorem, which is formulated directly in
terms of the one-factor isoperimetric profile.
In particular, neither completeness nor
orientability is required in the statement of
Theorem~\ref{thm:manifold} itself.
\end{remark}

The proof of Theorem~\ref{thm:manifold} constructs an even log-concave measure whose profile is $I$
and applies Theorem~\ref{thm:general} to that model. Ros's product
comparison \cite[Theorem~3.7]{Ros2005} then transfers the lower bound
back to $(M^n,g^{\oplus n},\mu^{\otimes n})$, while single-factor cylinders give the matching
upper bound. The classification of equality requires a further argument:
we trace equality through the successive fiber rearrangements rather
than infer it from the profile comparison alone. In particular, neither
log-concavity of the original weight nor a curvature lower bound on
$M$ is assumed in Theorem~\ref{thm:manifold}; they are replaced by the
stated conditions on the \emph{actual} profile. The one-dimensional
local-optimality and rigidity proofs are used only once.

For round spheres and Euclidean balls, Barthe
\cite[Sections~2.2 and 2.3]{Barthe2001}
proved that their Gaussian contact points are both $1/2$. Therefore, we obtain the following corollary.

\begin{corollary}[Products of round spheres and Euclidean balls]
\label{cor:sphere}
Let $S_R^d$ be the $d$-dimensional round sphere
and $B_R^d$ be the $d$-dimensional Euclidean ball,
equipped with their standard metrics and normalized
volume measures.
For each integer $d\ge1$, there exist
$\varepsilon_{\mathrm S,d},\varepsilon_{\mathrm B,d}>0$
such that, for every $R>0$ and $n\ge1$,
\begin{align}
I_{(S_R^d)^n}(a)
&=I_{S_R^d}(a)
&&\bigl(|a-1/2|\le\varepsilon_{\mathrm S,d}\bigr),
\label{eq:spherecentral}\\
I_{(B_R^d)^n}(a)
&=I_{B_R^d}(a)
&&\bigl(|a-1/2|\le\varepsilon_{\mathrm B,d}\bigr).
\label{eq:ballcentral}
\end{align}
All minimizers are cylinders over single-factor
isoperimetric regions: spherical caps in the sphere
case, and intersections with orthogonal balls or
their complements in the Euclidean ball case.
At half-volume, the latter reduce to half-balls. 
The intervals are uniform in $n$ and $R$,
but no uniformity in $d$ is asserted.
\end{corollary}

\begin{remark}
Since $(B_R^1)^n$ is exactly the $n$-cube with side-length $2R$, the quantitative result for products
of $B_R^1$ follows directly from the corresponding
cube result. By the standard reflection argument for flat tori
\cite[Remark~2.4]{Milman2026}, the isoperimetric
profiles of products of circles and Euclidean cubes
satisfy
$
I_{(S_R^1)^n}(a)
=\frac{1}{\pi R} I_{Q_n}(a), \,
  n\ge1, 0<a<1.
$
Consequently, the quantitative result for products
of circles follows directly from the corresponding
cube result.
Therefore, for $d=1$, one may take
$\varepsilon_{\mathrm S,1}
=\varepsilon_{\mathrm B,1}=9/625$, yielding
\begin{align}
I_{(S_R^1)^n}(a)
=\frac{1}{\pi R}, \qquad
I_{(B_R^1)^n}(a)
=\frac{1}{2R}, &&\bigl(|a-1/2|\le9/625=0.0144\bigr).
\label{eq:intervalquant}
\end{align} 
For $d=2$, using the similar estimate for cubes, we obtain 
\begin{align}
I_{(S_R^2)^n}(a)
&=R^{-1}\sqrt{a(1-a)}
&&\bigl(|a-1/2|\le\frac{1}{80}=0.0125\bigr),
\label{eq:sphere2central}\\
I_{(B_R^2)^n}(a)
&=I_{B_R^2}(a)
&&\bigl(|a-1/2|\le\frac{29}{2000}=0.0145\bigr).
\label{eq:diskcentral}
\end{align}
\end{remark}

Barthe \cite{Barthe2001} established half-volume optimality for products
of spheres. Corollary~\ref{cor:sphere} extends that single-volume
conclusion to a common interval by combining the present one-dimensional
theorem with the classical spherical isoperimetric model. The width is
allowed to depend on the fixed sphere dimension.

\subsection{Conjectures for the limiting profile}

\label{conjectures-section} 
The central open question on this topic is to give a complete characterization of the limiting isoperimetric profile $I_{\infty}$ for an arbitrary probability measure $\mu$.  
As proven in \cite{Yu2025LargeDeviations}, the limiting isoperimetric profile $I_{\infty}$ is identical to the isoperimetric profile for the infinite-dimensional product space $(M^\infty,d_\infty,\mu_\infty)$ in which the perimeter is defined via Minkowski content. 
Thus, a complete characterization of $I_{\infty}$ would not only yield a sharp dimension-free lower bound for all finite-dimensional isoperimetric profiles, but also give a complete characterization of the infinite-dimensional isoperimetric profile. 
Barthe \cite{Barthe2001} resolved this open question at Gaussian contact points, while our results in this paper resolve it on some nontrivial interval around the Gaussian contact points. 
In general, $K_{\mathrm{G}}\IG(a)\le I_\infty(a)\le\sqrt{K_{\mathrm{P}}}\IG(a)$; both constants are defined above. 
The upper bound here is from \cite{Yu2025LargeDeviations} and the limit case as $a\downarrow0$ is resolved in \cite{yu2026sharp}. 

\begin{theorem}[Upper bound and small-volume asymptotics \cite{Yu2025LargeDeviations,yu2026sharp}]\label{small-volume}
For a weighted manifold satisfying the assumptions in Remark \ref{rem:CD}, it holds that 
\[
I_{\infty} (a)\le\sqrt{K_{\mathrm{P}}} \IG(a)\qquad a\in[0,1],
\]
and 
\begin{equation}
I_{\infty} (a)\sim\sqrt{K_{\mathrm{LS}}}\IG(a)\qquad(a\downarrow0).\label{small-volume-asymptotic}
\end{equation}
Moreover, the upper bound $\sqrt{K_{\mathrm{P}}}\IG(a)$ is asymptotically attained
by the set 
\begin{equation}
\left\{x\in M^n: \frac{1}{\sqrt{n}} \sum^{n}_{i=1}g_n(x_{i}) \le t \right\} 
\end{equation}
for some $t$ such that the volume asymptotically approaches $a$, where $\{g_n\}$ are a sequence of functions asymptotically attaining  the Poincar\'e constant. 
\end{theorem}

For cubes, $K_{\mathrm{LS}}=K_{\mathrm{P}}=\pi^2$. Thus, in this case,
$I_{Q_\infty}(a)\le\pi\IG(a) $
and 
$I_{Q_\infty}(a)\sim\pi\IG(a)$. 
The upper bound $\pi\IG(a)$ is asymptotically attained
by 
\begin{equation}
\left\{x\in Q_{n}:\sqrt{\frac{2}{n}}\sum^{n}_{i=1}\cos(\pi x_{i})\le\Phi^{-1}(a)\right\}.
\end{equation}

We now pose a conjectural formula that extends
the central plateau theorem for cubes. Throughout this section, $v=\min(a,1-a)$.  
Let $\omega_{k}=\pi^{k/2}/\Gamma(1+k/2)$ be the volume of the Euclidean
unit ball in $\R^{k}$, and define 
\begin{equation}
c_{k}=\frac{k}{2}\omega^{1/k}_{k},\qquad B_{k}(v)=c_{k}v^{1-1/k}\quad(v>0).\label{corner-branches}
\end{equation}
In particular, $B_{1}(v)=1$, $B_{2}(v)=\sqrt{\pi v}$, and $B_{3}(v)=3(\pi/6)^{1/3}v^{2/3}$.
The remaining coefficients include  
\[
c_{4}=2(\pi^{2}/2)^{1/4},\qquad c_{5}=\frac{5}{2}(8\pi^{2}/15)^{1/5},\qquad c_{6}=3(\pi^{3}/6)^{1/6}.
\]
For the positive-orthant portion of the Euclidean unit ball in $k$ coordinates with center at the origin and with radius $r$,
extended cylindrically in the remaining coordinates, the  volume and
perimeter are 
\[
v=\frac{\omega_{k}r^{k}}{2^{k}},\qquad P=\frac{k\omega_{k}r^{k-1}}{2^{k}}=B_{k}(v).
\]
This interpretation requires $r\le1$, equivalently $v\le\omega_{k}/2^{k}$,
so that the ball does not cross the opposite cube faces. The factors
$2^{-k}$ arise because the orthant occupies $2^{-k}$ of a ball,
and the coordinate-plane boundaries are excluded from perimeter.
A product with unused coordinates preserves both this volume and perimeter.

Based on the complete solution for low dimensional cubes (i.e., dimensions $n=1,2,3$ \cite{BrezisBruckstein2021,DengLiLiu2026,OpenAI2026Torus}) and the small-volume asymptotic solutions, we make the following conjecture. 

\begin{conjecture}[Seven-branch limiting profile]
\label{profile-conjecture} For every $a\in(0,1)$, 
\begin{equation}
I_{Q_\infty}(a)=\min\{\pi\IG(a),B_{1}(v),B_{2}(v),B_{3}(v),B_{4}(v),B_{5}(v),B_{6}(v)\},\quad v=\min(a,1-a).\label{seven-branch-conjecture}
\end{equation}
At the endpoints, $I_{Q_\infty}(0)=I_{Q_\infty}(1)=0$. 
\end{conjecture} 

Solving $B_{k}(\tau_{k})=B_{k-1}(\tau_{k})$ gives
\[
\tau_{k}=\left(\frac{k-1}{k}\right)^{k(k-1)}\frac{\omega^{\,k}_{k-1}}{\omega^{\,k-1}_{k}}.
\]
In particular, 
\begin{align*}
\tau_{6} &=\frac{5^{24}}{2^{7}3^{31}\pi^{3}}\simeq0.024314291553,
 &\tau_{5} &=\frac{2^{23}3^{4}\pi^{2}}{5^{16}}\simeq0.043949566442,\\
\tau_{4} &=\frac{6561}{8192\pi^{2}}\simeq0.081148472397,
 &\tau_{3} &=\frac{4\pi}{81}\simeq0.155140377955,\\
\tau_{2} &=\frac1\pi\simeq0.318309886184.
\end{align*}
The Gaussian-to-six-dimensional crossing is denoted by $\tau_{\mathrm{G}}$
and satisfies 
\begin{equation}
\pi\IG(\tau_{\mathrm{G}})=B_{6}(\tau_{\mathrm{G}}),\qquad\tau_{\mathrm{G}}\simeq0.015890647092.\label{gaussian-crossing}
\end{equation}
The conjectural $I_{Q_\infty}$ is explicitly given in Table \ref{candidate-table}. 

\begin{table}[t]
\centering %
\begin{tabular}{lll}
\toprule 
Volume range  & Proposed value  & Branch\tabularnewline
\midrule 
$0<v\le\tau_{\mathrm{G}}$  & $\pi\IG(v)$  & Gaussian\tabularnewline
$\tau_{\mathrm{G}}\le v\le\tau_{6}$  & $B_{6}(v)$  & Six-dimensional orthant ball\tabularnewline
$\tau_{6}\le v\le\tau_{5}$  & $B_{5}(v)$  & Five-dimensional orthant ball\tabularnewline
$\tau_{5}\le v\le\tau_{4}$  & $B_{4}(v)$  & Four-dimensional orthant ball\tabularnewline
$\tau_{4}\le v\le\tau_{3}$  & $B_{3}(v)$  & Three-dimensional orthant ball\tabularnewline
$\tau_{3}\le v\le1/\pi$  & $\sqrt{\pi v}$  & Quarter-disk cylinder\tabularnewline
$1/\pi\le v\le1/2$  & $1$  & Coordinate slab\tabularnewline
\bottomrule
\end{tabular}\caption{Predicted branches of Conjecture~\ref{profile-conjecture}. Reflection
about $a=1/2$ gives the other half of the profile.}
\label{candidate-table} 
\end{table}

For $0<a<\tau_{\mathrm{G}}$, the upper bound $\pi\IG(a)$ is strictly
smaller than $B_{k}(a)$ for every $k\ge1$. Consequently, there exist
sets in sufficiently high-dimensional cubes whose perimeters are strictly
smaller than those of cylindrical orthant balls. This provides an
explicit confirmation of the phenomenon established in \cite[Theorem~1.6]{Milman2026},
namely, that cylindrical orthant balls need not be isoperimetric minimizers. 

\subsection{Proof idea}

The proof is inspired by Glaudo's work
\cite{Glaudo2024} and combines two complementary
mechanisms, termed \emph{local optimality} and
\emph{rigidity}. Local optimality establishes
perimeter minimality of coordinate half-spaces
among competitors whose fiber-volume functions
are sufficiently close to their means in $L^1$,
while rigidity ensures that every sufficiently
competitive set is close in symmetric difference
to a coordinate half-space of the same volume.
This symmetric-difference closeness implies
the $L^1$ closeness required for local optimality.
Thus, rigidity reduces the problem to a
neighborhood in which local optimality rules
out any perimeter improvement, establishing
global isoperimetric optimality uniformly
over all product dimensions.
A similar local-optimality--rigidity principle
underlies the author's earlier work on discrete
cubes \cite{Yu2026LocalOptimality}, although
the specific arguments are quite different.

To establish rigidity, similarly to  Glaudo's work, we use the coordinatewise
Gaussian transport $T=F^{-1}\circ\Phi$.
For a set $E\subset\Omega^n$ of measure $a$,
put $A=(T^{\otimes n})^{-1}(E)$.
The transport area formula yields the exact
perimeter-excess decomposition
\[
 \kappa^{-1} P_{\mu^{\otimes n}}(E)-I_G(a) 
= R+ \Delta,
\]
where $R=\int_{\partial^*A}(W-1)\,d\sigma_A\ge0$ is  the nonnegative transport excess,  
$\Delta=P_\gamma(A)-I_G(a)\ge0$ is the Gaussian isoperimetric deficit, and 
$W\ge1$ is the transport-induced surface weight.
For any competitor with
$P_{\mu^{\otimes n}}(E)\le I(a)$,
the Gaussian contact condition implies
\[
0\le R+ \Delta
\le \kappa^{-1}  I(a)  -I_G(a)
\longrightarrow0
\qquad(a\to b).
\]
Since $\Delta\to 0$, the quantitative Gaussian stability theorem of Barchiesi, Brancolini, and Julin \cite{BarchiesiBrancoliniJulin2017}, as also used by Glaudo \cite{Glaudo2024}, then shows that the transported set is close to a Gaussian half-space. To further localize its normal, we develop a bounded test-field transfer method that exploits the smallness of the Gaussian deficit and the transport excess to force the normal toward a coordinate direction. Finally, by comparing equal-volume Gaussian half-spaces with different normals, we identify a coordinate half-space that is close to the original set in symmetric difference. The argument is uniform over all product dimensions.
This symmetric-difference closeness also
ensures that the corresponding fiber-volume
function is close to its mean in $L^1$,
as required for local optimality.

For local optimality, the main difficulty is
to control the Jensen deficit
$I(a)-\int I(u)\,d\nu$ associated with the
concave one-factor isoperimetric profile $I$,
where $u$ is a fiber-volume function with
mean $a=\int u\,d\nu$.
At a Gaussian contact point $b$, we establish
the strict spectral-gap margin
\[
-I(b)I''(b)\le\kappa^2<K_{\mathrm P}
\]
for non-Gaussian measures.
This margin allows us to combine semigroup
estimates and Fenchel duality to show that
the BV perimeter excess
$\mathcal E_I(u)=\mathcal F_I(u)-\int I(u)\,d\nu$
controls the variance of $u$ when its
oscillation is sufficiently small, uniformly
over all product dimensions.
A clipping argument separates small oscillations
from larger deviations occurring on sets of
small measure. Small-support isoperimetric
inequalities and the BV chain rule show that
the perimeter excess also controls the
contribution of these larger deviations.
Meanwhile, Taylor expansion of $I$ around $a$,
together with the cancellation
$\int(u-a)\,d\nu=0$, allows us to bound the
Jensen deficit by a multiple of
$\operatorname{Var}_\nu(u)$, up to a tail
contribution. The variance of $u$ is in turn
controlled by the variance of its clipped
version and the same tail contribution.
Thus, the small-oscillation estimate and the
tail estimate together yield
\[
I(a)-\int I(u)\,d\nu
\le (1-\tau)\mathcal E_I(u)
\]
for some $\tau>0$, provided that $a$ is
sufficiently close to $b$ and $u$ is
sufficiently close to $a$ in $L^1$.
Consequently,
\[
\mathcal F_I(u)-I(a)
\ge\tau\mathcal E_I(u)\ge0,
\]
uniformly over all product dimensions,
establishing the desired local-optimality
inequality. Since rigidity places every potentially
improving competitor within this neighborhood,
local optimality suffices to establish
global isoperimetric optimality.

Using the proof idea described above, we first establish the analytic estimates
and the Gaussian-contact theorem for
one-factor measures with full support on
$\mathbb R$. Finite-support measures are
then treated using a smooth full-support
profile minorant and a coordinatewise
contracting transport, preserving Gaussian
contact and the classification of minimizers.
We further extend the theorem to weighted
Riemannian manifolds by realizing their
one-factor isoperimetric profiles through
one-dimensional log-concave models and
applying Ros's product comparison.
The equality cases are obtained by tracing
equality through successive fiber
rearrangements and BV slicing.
The applications to products of spheres and
Euclidean balls then follow from the general
manifold theorem.

For the explicit central interval in cubes,
we develop sharper quantitative versions
of local optimality and rigidity.
For rigidity, a bounded test field with a
quadratic core, combined with Gaussian column
estimates and affine cancellation, yields
a quantitative coordinate-localization bound.
The strong concavity of the Gaussian profile
then converts this bound into a fiber-variance
estimate sufficient for local optimality.
For the latter, we lift the cube problem
to a product of two-dimensional spheres,
using their positive curvature to obtain
explicit semigroup estimates, and develop
a refined tail-compensation argument that
handles empty and full fibers without
artificial endpoint penalties.
Together, these estimates establish the
dimension-independent central half-width
$9/625$. The circle-product result follows
from the cube--circle profile identity.

The qualitative Gaussian-contact theorem and the quantitative cube and circle results share the same fiber-reduction framework, with sharper estimates developed in Appendix A for the latter. Section~\ref{sec:tools} establishes Gaussian transference
and the weighted BV fiber inequality; Section~\ref{sec:tools} also records the one-dimensional model
construction and its Riemannian fiber comparison.  Section~\ref{sec:local}
proves dimension-free local optimality; and Section~\ref{sec:rigid} gives dimension-free rigidity
of low-perimeter competitors near coordinate half-spaces.
Section~\ref{sec:proofs}
closes the contact-level local-optimality--rigidity argument and then
derives the manifold theorem and the sphere application. Appendix~\ref{app:circlequant} proves the explicit interval
constant for the cube and circle cases.  

\section{Gaussian transference and BV slicing}\label{sec:tools}
In the analytic arguments below, the one-dimensional measure $\mu$
is assumed to satisfy \eqref{eq:conditions}, unless a different hypothesis
is stated explicitly. By \eqref{eq:concavityequivalence}, its potential
is therefore convex. 
We begin with two estimates used throughout the paper:
a Gaussian isoperimetric lower bound valid in all product
dimensions and a sharp BV fiber-perimeter inequality.
The latter retains the Euclidean square-root structure
combining fiberwise perimeter and transverse variation,
including the singular BV contribution. Together, these
estimates provide a common framework for the general
log-concave setting and the sharp uniform-cube analysis.
We use standard notation from the theory of
BV functions; see
\cite[Chapter~3]{AmbrosioFuscoPallara2000}.

Differentiation of the quantile representation \eqref{eq:I}, with
the abbreviation $t=F^{-1}(a)$, gives
\begin{equation}\label{eq:profilederivatives}
  I'(a)=-V'(t),\qquad I''(a)=-\frac{V''(t)}{I(a)}.
\end{equation} 
For $\nu=\mu^{\otimes m}$, where
$d\mu=f(x) \,dx =e^{-V(x)}\,dx$, let
$f_m(x)=\prod_{i=1}^m f(x_i)$. 
Define
\[
BV(\nu)=
\left\{
u\in L^1(\nu)\cap BV_{\mathrm{loc}}(\Omega^m):
\int_{\Omega^m}f_m\,d|Du|<\infty
\right\}.
\]
For $u\in BV(\nu)$, write
\begin{equation}\label{eq:Lebesgue-decomposition}
Du=\nabla u\,\mathcal L^m+D^s u,
\end{equation}
where $\mathcal L^m$ is the $m$-dimensional
Lebesgue measure and $D^s u$ is the singular
part of the distributional derivative. The relative weighted
total variations are defined by
\[
|Du|_\nu
=\int_{\Omega^m}f_m\,d|Du|,
\qquad
|D^s u|_\nu
=\int_{\Omega^m}f_m\,d|D^s u|.
\]
The continuous profile $B$
will be used as a comparison kernel, while $I$ denotes the actual
one-dimensional profile under \eqref{eq:conditions}. For a general
base manifold we use $I$ for its actual profile and $K$ for an
abstract model profile, so that these roles are not conflated.

For a set $A$ of finite perimeter, write
$\partial^*A$ for its reduced boundary and
$\mathbf n_A$ for its measure-theoretic outward
unit normal. For the weighted product measure
$\nu=\mu^{\otimes m}$ with positive density
$f_m$, define the relative weighted perimeter
and total variation by
\[
P_\nu(A)
=\int_{\partial^*A\cap\Omega^m}
f_m\,d\mathcal H^{m-1},
\qquad
|Du|_\nu
=\int_{\Omega^m}f_m\,d|Du|,
\]
where $\mathcal H^{m-1}$ denotes the $(m-1)$-dimensional Hausdorff measure.

All boundaries lying on $\partial\Omega^m$ are excluded. When the
support is unbounded, the corresponding identities are first applied
on compact subsets and then passed to the limit; the estimates are
monotone under exhaustion. We refer to
\cite{AmbrosioFuscoPallara2000,Maggi2012} for the BV area, slicing,
coarea, and chain-rule theorems used in these passages.

We collect several estimates used in the proofs. 

\begin{lemma}[Transfer and small-support control \cite{Barthe2001}]\label{lem:transfer}
Under the Gaussian comparison $I\ge B$, for all $m\ge1$, it holds that
\begin{align}
 I_{\nu}(s)&\ge B(s)\quad(0\le s\le1),\label{eq:gausslower}\\
 B\big(\int u\dd\nu\big)
 &\le\int\sqrt{B(u)^2+|\nabla u|^2}\dd\nu+|D^su|_\nu,
 \quad 0\le u\le1.\label{eq:bobkovBV}
\end{align}
Moreover, if $w\ge0$ is in $BV(\nu)$ and
$\nu\{w>0\}\le q\le1/2$, then
\begin{equation}\label{eq:smallsupport}
 |Dw|_\nu\ge H(q)\int w\dd\nu,
 \qquad H(q)=\frac{B(q)}q\longrightarrow\infty\quad(q\downarrow0).
\end{equation}
\end{lemma}

\begin{proof}
The Gaussian comparison for product measures follows
from \cite[Corollary~3(i)]{Barthe2001}.
The functional inequality for locally Lipschitz
functions follows from
\cite[Theorem~1]{Barthe2001}, and its extension
to $BV(\nu)$ is obtained by standard area-strict
approximation, retaining the singular variation;
see \cite[Chapter~3]{AmbrosioFuscoPallara2000}.
Finally, the weighted BV coarea formula and the
first inequality yield
\[
\begin{aligned}
|Dw|_\nu
&=\int_0^\infty
P_\nu(\{w>t\})\,dt \ge\int_0^\infty
B(\nu\{w>t\})\,dt\\
&\ge\frac{B(q)}q
\int_0^\infty\nu\{w>t\}\,dt =H(q)\int w\,d\nu,
\end{aligned}
\]
where the second inequality follows from the concavity
of $B$ and $B(0)=0$.
Since $B(q)=\kappa I_{\mathrm G}(q)$,
we have $H(q)\to\infty$ as $q\downarrow0$.
\end{proof}

For $E\subset\Omega\times\Omega^m$ and
$y\in\Omega^m$, let
\[
E_y:=\{x\in\Omega:(x,y)\in E\}
\]
denote the fiber of $E$ over $y$.
For $u\in BV(\nu;[0,1])$ define the BV fiber-perimeter functional and
 the BV fiber-perimeter excess by
\begin{align}
 \cF_I(u)&=\int\sqrt{I(u)^2+|\nabla u|^2}\dd\nu
               +|D^su|_\nu,\label{eq:F}\\
 \cE_I(u)&=\cF_I(u)-\int I(u)\dd\nu.\label{eq:E}
\end{align}

\begin{lemma}[Weighted BV fiber reduction]\label{lem:fiber}
For any finite-perimeter $E\subset\Omega\times\Omega^m$, put
$u(y)=\mu(E_y)$ and $p(y)=P_\mu(E_y)$. Then
$u\in BV(\nu;[0,1])$, $p\ge I(u)$ almost everywhere, and
\begin{equation}\label{eq:vectorBV}
 P_{\mu\otimes\nu}(E)
 \ge\int\sqrt{p(y)^2+|\nabla u(y)|^2}\dd\nu(y)
             +|D^su|_{\nu}\ge\cF_I(u).
\end{equation}
\end{lemma}

\begin{proof}
This is a standard consequence of weighted BV
slicing and the Gauss--Green formula; see
\cite[Theorem~2.6]{Fusco2015} and
\cite[Proposition~1.3]{FuscoPallara2017}.
Writing $d\nu(y)=f_m(y)\,dy$, weighted BV slicing
and the Gauss--Green formula show that the total
variation of the projected vector measure
$
\bigl(pf_m\,\mathcal L^m,-f_m\,Du\bigr)
$
is bounded by $P_{\mu\otimes\nu}(E)$.
Decomposing this measure into its absolutely
continuous and singular parts as given in \eqref{eq:Lebesgue-decomposition} yields
\[
\int\sqrt{p(y)^2+|\nabla u(y)|^2}\,d\nu(y)
+|D^su|_\nu.
\]
The second inequality in \eqref{eq:vectorBV}
follows from $p\ge I(u)$.
\end{proof}

\begin{lemma}[One-dimensional realization of a concave profile \cite{Bobkov1996}]
\label{lem:modelprofile}
Let $J:(0,1)\to(0,\infty)$ be $C^2$, concave, and symmetric about
$1/2$, and prescribe $J(0)=J(1)=0$. There is an even log-concave
probability measure $\mathfrak m_J$ on a symmetric interval $(-\ell_J,\ell_J)$,
with $C^2$ potential, such that $I_{\mathfrak m_J}=J$. More precisely, set
\begin{equation}\label{eq:modelquantile}
 Q_J(a)=\int_{1/2}^a\frac{ds}{J(s)},\qquad
 \ell_J=\int_0^{1/2}\frac{ds}{J(s)},\qquad F_J=Q_J^{-1}.
\end{equation}
Then $d\mathfrak m_J(t)=f_J(t)\,dt$, where $f_J=J\circ F_J$.
\end{lemma}
\begin{proof}
The map $Q_J$ increases from $-\ell_J$ to $\ell_J$ and satisfies
$Q_J(1-a)=-Q_J(a)$. Its inverse obeys $F_J'=J(F_J)$, so
$f_J=F_J'$ is an even probability density, with
$f_J(0)=J(1/2)$. For $V_J=-\log(f_J/J(1/2))$,
\begin{equation}\label{eq:modelpotential}
 V_J'(t)=-J'(F_J(t)),\qquad
 V_J''(t)=-J''(F_J(t))J(F_J(t))\ge0.
\end{equation}
Thus $V_J$ is an even $C^2$ convex potential. Half-lines minimize
one-dimensional perimeter for this symmetric log-concave measure,
so $I_{\mathfrak m_J}(a)=f_J(F_J^{-1}(a))=J(a)$ in the interior. Empty and
full sets give the prescribed endpoint values. No continuity at
$0$ or $1$ is needed, which includes the constant profile of a
uniform interval. 
This classical profile realization is due to
Bobkov \cite{Bobkov1996}. 
\end{proof}

The next lemma enables us to deduce the finite-support case from the full-support case. The key is to
construct a smooth log-concave measure on
$\mathbb R$ whose isoperimetric profile lies
below $I$, retains the Gaussian comparison,
and agrees with $I$ near the contact points.
The associated monotone transport is
$1$-Lipschitz, allowing us to transfer the
dimension-free isoperimetric identity from
the full-support model to the original measure.

\begin{lemma}[A full-support comparison preserving a Gaussian contact]
\label{lem:fullsupportminorant}
Suppose $0<L<\infty$ and $I$ is the profile of an even
log-concave measure as in \eqref{eq:measure}, satisfying
\eqref{eq:conditions}--\eqref{eq:contact}. There exist a positive,
symmetric, concave $C^2$ function $J$ on $(0,1)$ and $\zeta>0$
such that
\begin{equation}\label{eq:fullsupportminorant}
 \kappa\IG(a)\le J(a)\le I(a)\quad(0<a<1),
 \qquad J(a)=I(a)\quad\text{if }|a-b|<\zeta
 \text{ or }|a-(1-b)|<\zeta.
\end{equation}
Moreover, $J$ is the isoperimetric profile of a non-Gaussian,
even log-concave probability measure $\widehat\mu$ on $\mathbb R$
with a $C^2$ convex potential. The monotone quantile transport
$\Theta=F_\mu^{-1}\circ F_{\widehat\mu}:\mathbb R\to(-L,L)$
satisfies $\Theta_\#\widehat\mu=\mu$ and
\begin{equation}\label{eq:fullsupporttransport}
 0<\Theta'(x)=\frac{J(F_{\widehat\mu}(x))}
 {I(F_{\widehat\mu}(x))}\le1.
\end{equation}
\end{lemma}
\begin{proof}
The finite support condition and the quantile formula give
\[
 L=\int_0^{1/2}\frac{da}{I(a)}<\infty.
\]
We claim that $I(a)/\IG(a)\to\infty$ as $a\downarrow0$.
Indeed, otherwise there are $C<\infty$ and $a_j\downarrow0$
with $I(a_j)\le C\IG(a_j)$. Concavity and $I(0)=0$
imply $I(s)/s\le I(a_j)/a_j$ for
$a_j\le s\le\sqrt{a_j}$. Consequently,
\[
 L\ge\int_{a_j}^{\sqrt{a_j}}\frac{ds}{I(s)}
 \ge\frac{a_j}{2I(a_j)}\log\frac1{a_j}
 \ge\frac{a_j}{2C\IG(a_j)}\log\frac1{a_j}
 \longrightarrow\infty,
\]
since $\IG(a)/a\sim\sqrt{2\log(1/a)}$. This contradicts $L<\infty$. 
By symmetry the same holds near $a=1$.

Choose $A>\kappa$ and
$0<\sigma<A-\kappa$. Let $\rho_\sigma$
be a smooth even probability density
supported in $(-\sigma,\sigma)$, and define
\[
J_\sigma(a)
=\int_{-\sigma}^{\sigma}
\min\{I(a),(A+t)I_G(a)\}
\rho_\sigma(t)\,dt.
\]
Since $A+t>\kappa$ and $I\ge\kappa I_G$,
we immediately obtain
\[
\kappa I_G\le J_\sigma
\le\min\{I,A I_G\}.
\]
Moreover, $J_\sigma$ is symmetric and concave,
being an average of symmetric concave
functions.

To verify smoothness, write
\[
J_\sigma(a)=I_G(a)M_\sigma(r(a)),
\quad
r(a)=\frac{I(a)}{I_G(a)},
\]
where
\[
M_\sigma(u)
=\int_{-\sigma}^{\sigma}
\min\{u,A+t\}\rho_\sigma(t)\,dt.
\]
Since $M_\sigma''(u)=-\rho_\sigma(u-A)$,
we have $J_\sigma\in C^2(0,1)$.

At each Gaussian contact point,
$r(b)=\kappa<A-\sigma$, so continuity implies
$J_\sigma=I$ in a neighborhood of $b$
(and likewise of $1-b$).
Finally, $J_\sigma\le A I_G$ ensures
\[
\int_0^{1/2}\frac{da}{J_\sigma(a)}
=\infty,
\]
so the corresponding log-concave model
has full support on $\mathbb R$.
\end{proof}

The next lemma will be used to reduce the product weighted manifold case to the product weighted real line case. 

\begin{lemma}[Ros's product comparison theorem \cite{Ros2005}]
\label{lem:roscomparison}
Let $(X,g_1,\mathfrak m_1)$ and
$(Y,g_2,\mathfrak m_2)$ be smooth weighted
Riemannian probability spaces, with relative
perimeters if boundaries are present.
If $I_{\mathfrak m_1}\ge J$ and $\mathfrak m_J$ is
the model measure in Lemma~\ref{lem:modelprofile},
then
\begin{equation}\label{eq:roscomparison}
I_{\mathfrak m_1\otimes\mathfrak m_2}(a)
\ge I_{\mathfrak m_J\otimes\mathfrak m_2}(a),
\qquad 0\le a\le1.
\end{equation}
More precisely, replacing each $X$-fiber of a
finite-perimeter set by a left half-interval
of $\mathfrak m_J$ with the same mass preserves volume
and does not increase perimeter.
\end{lemma}

\begin{proof}
This is Ros's product comparison principle
\cite[Section~3.3, Theorem~3.7]{Ros2005};
see also \cite[Theorem~2.18]{Milman2026}.
For the weighted BV perimeter formulation used
here, let $u(y)=\mathfrak m_1(E_y)$.
The weighted BV fiber inequality and
$I_{\mathfrak m_1}\ge J$ give
$
P_{\mathfrak m_1\otimes\mathfrak m_2}(E)
\ge \mathcal F_J(u).
$
The corresponding fiberwise rearrangement $E^*$
in $\mathfrak m_J\otimes\mathfrak m_2$ satisfies
$
P_{\mathfrak m_J\otimes\mathfrak m_2}(E^*)
=\mathcal F_J(u),
$
by the area formula for model half-intervals,
including the singular BV contribution.
This proves the assertion.
\end{proof}
 
\section{Dimension-free local optimality}
\label{sec:local}
In this section the diffusion-semigroup argument is carried out only
for full-support one-dimensional measures ($L=\infty$).
The case $L<\infty$ in Theorem~\ref{thm:general} is recovered later
by Lemma~\ref{lem:fullsupportminorant} and a contracting transport;
the separate quantitative cube argument in Appendix~\ref{app:circlequant}
uses a compact boundaryless sphere instead.
The local-optimality argument shows that the BV perimeter excess $\mathcal E_I(u)$ arising from variations in fiber masses controls the Jensen deficit $I(a)-\int I(u)\,d\nu$ of the one-dimensional isoperimetric profile.
In proving this, a spectral-gap estimate
controls small-amplitude fluctuations, while a tail-compensation  estimate
handles large deviations and singular BV derivatives. Both estimates
are uniform in the number of product factors, and neither presupposes
that the set boundary is a smooth graph.

We first establish a strict spectral-gap margin at a Gaussian contact point; see \eqref{eq:contactgap}. This margin allows us to use the spectral-gap contraction of the diffusion semigroup to yield a dimension-free BV perimeter–variance estimate, as done later.
\begin{lemma}[Strict spectral gap at a Gaussian contact]\label{lem:spectralrigid}
Under \eqref{eq:conditions}, the optimal Poincar\'e constant
$K_{\mathrm{P}}$ satisfies $K_{\mathrm{P}}\ge\kappa^2$, and moreover,
\begin{equation}\label{eq:gaprigid}
 K_{\mathrm{P}}=\kappa^2\quad\Longleftrightarrow\quad
 \mu=N(0,\kappa^{-2}).
\end{equation}
If in addition \eqref{eq:contact} holds at $b\in(0,1)$, then
\begin{equation}\label{eq:strictgap}
 -I(b)I''(b)=V''(F^{-1}(b))\le\kappa^2.
\end{equation}
Consequently, unless $\mu$ is Gaussian,
\begin{equation}\label{eq:contactgap}
 -I(b)I''(b)<K_{\mathrm{P}}.
\end{equation}
\end{lemma}
\begin{proof}
Since $I-B\ge0$ and $I(b)=B(b)$ at an interior point,
$I'(b)=B'(b)$ and $I''(b)\ge B''(b)$. As
$\IG''(s)=-1/\IG(s)$, multiplication by $I(b)=\kappa\IG(b)$ gives
$-I(b)I''(b)\le\kappa^2$. Moreover,
$I''(b)=-V''(F^{-1}(b))/I(b)$, proving the curvature identity. 
Let $T=F^{-1}\circ\Phi$ be the monotone Gaussian
transport. Since $F(T(z))=\Phi(z)$, we have
\[
T'(z)
=\frac{\phi(z)}{f(T(z))}
=\frac{I_G(\Phi(z))}{I(\Phi(z))}
\le\frac1\kappa,
\]
where the last inequality follows from
$I\ge\kappa I_G$.

Write $a(z)=T'(z)^{-2}\ge\kappa^2$. 
Let $H^1(\mu)$ be the weighted Sobolev
space of functions $h\in L^2(\mu)$ whose weak
derivatives $h'$ belong to $L^2(\mu)$.
For $h\in H^1(\mu)$ and $u=h\circ T$, a change of variables yields
\[
 \operatorname{Var}_\mu(h)=\operatorname{Var}_\gamma(u),
 \qquad
 \int|h'|^2\,d\mu=\int a(z)|u'(z)|^2\,d\gamma(z).
\]
The sharp Gaussian Poincar\'e inequality (with optimal constant $1$) yields $K_{\mathrm{P}}\ge\kappa^2$.

Suppose equality holds. Choose a sequence $u_j$ with Gaussian mean
zero and variance one such that
\[
 \int a(z)|u_j'(z)|^2\,d\gamma(z)\longrightarrow\kappa^2.
\]
It follows that
\begin{equation}\label{eq:gapminseq}
 \int|u_j'|^2\,d\gamma\to1,
 \qquad \int(a-\kappa^2)|u_j'|^2\,d\gamma\to0.
\end{equation}

The Ornstein--Uhlenbeck operator $-L=-\partial_z^2+z\partial_z$ has  $0,1,2,...$ as eigenvalues and  probabilists' Hermite polynomials as eigenfunctions. The functions $1,z$ are eigenfunctions corresponding to eigenvalues $0,1$.
Write $u_j=b_jz+r_j$ with
$r_j\perp\{1,z\}$ in $L^2(\gamma)$. Then, 
\begin{equation}\label{eq:sos}
b_j^2+\int r_j^2\,d\gamma=1.
\end{equation}
The second positive eigenvalue is $2$, and we have
\begin{equation}\label{eq:sum-of-derivative}
\int|r_j'|^2\,d\gamma
\ge2\int r_j^2\,d\gamma.
\end{equation}
Moreover, Gaussian integration by parts gives
$\int r_j'\,d\gamma=\int z r_j\,d\gamma=0$. Consequently,
\[
\int|u_j'|^2\,d\gamma-1 =\int|r_j'|^2\,d\gamma-\int r_j^2\,d\gamma \ge\frac12\int|r_j'|^2\,d\gamma.
\]
By \eqref{eq:gapminseq}, it follows that
$\int|r_j'|^2\,d\gamma\to0$ and further by \eqref{eq:sos} and \eqref{eq:sum-of-derivative},  $b_j^2\to1$.
Since $u_j'=b_j+r_j'$, Cauchy--Schwarz yields
\[
\begin{aligned}
\int\bigl||u_j'|^2-1\bigr|\,d\gamma
&\le |b_j^2-1|
 +2|b_j|\|r_j'\|_{L^2(\gamma)}
 +\|r_j'\|_{L^2(\gamma)}^2
 \longrightarrow0.
\end{aligned}
\]
Hence $|u_j'|^2\to1$ in $L^1(\gamma)$.

If $a>\kappa^2$ on a set of positive Gaussian
measure, then there exist $\delta>0$ and a measurable
set $B$ with $\gamma(B)>0$ such that
$a-\kappa^2\ge\delta$ on $B$. Since
$|u_j'|^2\to1$ in $L^1(\gamma)$,
\[
\liminf_{j\to\infty}
\int(a-\kappa^2)|u_j'|^2\,d\gamma
\ge\delta\gamma(B)>0,
\]
contradicting \eqref{eq:gapminseq}.
Thus $a\equiv\kappa^2$ by continuity. That is,  $T'\equiv1/\kappa$, and symmetry
implies $T(z)=z/\kappa$. This is precisely the stated Gaussian law.
The converse follows from Gaussian Poincar\'e. 
\end{proof}

Throughout, we denote the essential oscillation of a function $w\in BV(\nu)$ as 
\[
\operatorname{ess\,osc} w
:=\operatorname*{ess\,sup}_{\nu}w
-\operatorname*{ess\,inf}_{\nu}w.
\]
As promised, based on the spectral-gap margin established above, we now use the spectral-gap contraction of the diffusion semigroup to prove a dimension-free coercivity estimate for the BV perimeter excess for graphs of small-oscillation BV functions.  This is the central step in the local optimality part. 

\begin{lemma}[Small-oscillation BV perimeter–variance estimate]
\label{lem:spectralBV}
Let $\mu$ be a log-concave probability measure
of the form \eqref{eq:measure} with $L=\infty$,
set $\nu=\mu^{\otimes m}$, and fix $c>0$.
Assume that $\mu$ has Poincar\'e spectral gap
$\lambda>0$. For every
$0<K<\lambda/(2c)$, there exists
$r_0=r_0(\lambda,c,K)>0$, independent of $m$,
such that every $w\in BV(\nu)$ with
$\operatorname{ess\,osc}w\le r_0$ satisfies
\begin{equation}\label{eq:spectralBV}
\mathcal Q_c(w)
:=
\int\bigl(\sqrt{c^2+|\nabla w|^2}-c\bigr)\,d\nu
+|D^s w|_\nu
\ge K\operatorname{Var}_\nu(w).
\end{equation}
\end{lemma}

\begin{proof}
We first prove the estimate for $w\in H^1(\nu)$,
where the Dirichlet-form integration-by-parts
identity applies directly.  The BV extension is
given at the end.

Let $P_t$ be the conservative symmetric diffusion semigroup
on $\mathbb R^m$ with generator
\[
 \mathcal L=\Delta-\sum_{i=1}^m V'(x_i)\partial_i.
\]
Since $V''\ge0$, the product diffusion satisfies
$\mathrm{CD}(0,\infty)$. 
Let $w\in H^1(\nu)$. The reverse Poincar\'e inequality
\cite[Theorem~4.7.2(iv)]{BakryGentilLedoux2014} gives
\begin{equation}\label{eq:reversePoincare}
 2t|\nabla P_tw|^2\le P_t(w^2)-(P_tw)^2.
\end{equation}
This estimate has a dimension-free constant.  
We write $\nu w:=\int w\,d\nu$ for the
mean of $w$ under $\nu$. Define
\[
 R_Tw=\int_0^T P_t(w-\nu w)\,dt.
\] 
Since $\operatorname{ess\,osc} w\le r$, the right side of \eqref{eq:reversePoincare} is at most
$r^2/4$.  
Hence \eqref{eq:reversePoincare} implies
$\|\nabla P_tw\|_\infty\le r/\sqrt{8t}$.
Integrating over $t\in(0,T)$ gives
\begin{equation}\label{eq:RTgradient}
\|\nabla R_Tw\|_\infty
\le r\sqrt{T/2}.
\end{equation}

Write $g=w-\nu w$.
Let $A=-\mathcal L$ denote the nonnegative
self-adjoint generator on $L^2(\nu)$.
By strong continuity of $P_t$ and the semigroup property,
\begin{align}\label{eq:LR}
\mathcal L R_Tw
=\lim_{h\downarrow0}\frac{P_hR_Tw-R_Tw}{h}
=P_Tg-g \in L^2(\nu).
\end{align} 
For $w\in H^1(\nu)$, the Dirichlet-form
integration-by-parts identity 
and the spectral-gap contraction $\|P_tg\|_2 \le e^{-\lambda t} \|g\|_2$
\cite[Theorem~4.2.5]{BakryGentilLedoux2014} 
yield
\begin{align}
\int\nabla w\cdot\nabla R_Tw\,d\nu & =\langle w,AR_Tw\rangle =\langle g,(I-P_T)g\rangle \ge(1-e^{-\lambda T})
\operatorname{Var}_\nu(w).
\label{eq:RTfirst}
\end{align}
By Minkowski's integral inequality and the
spectral-gap contraction,
\[
\|R_Tw\|_2
\le\int_0^T\|P_tg\|_2\,dt
\le\frac{1-e^{-\lambda T}}{\lambda}\|g\|_2.
\]
Consequently, by integration-by-parts, Cauchy-Schwarz inequality, Minkowski's inequality,  and the
spectral-gap contraction,
\begin{align}
\int|\nabla R_Tw|^2\,d\nu
&=\langle R_Tw,g-P_Tg\rangle \le\|R_Tw\|_2
  (\|g\|_2+\|P_Tg\|_2)
\nonumber\\
&\le\frac{1-e^{-2\lambda T}}{\lambda}
  \operatorname{Var}_\nu(w)
 \le\lambda^{-1}\operatorname{Var}_\nu(w).
\label{eq:RTsecond}
\end{align}

Set
\[
X=\frac{\lambda}{c}\nabla R_Tw,
\qquad
s=\frac{\lambda r}{c}\sqrt{\frac T2}.
\]
By \eqref{eq:RTgradient}, we have
$|X| \le s<1$, provided that
$r>0$ is sufficiently small (for fixed $T$).
The Cauchy--Schwarz
inequality applied to
$(p,c)$ and $(X,\sqrt{1-|X|^2})$ yields
\begin{equation}\label{eq:BVFenchel}
\sqrt{c^2+|p|^2}-c
\ge p\cdot X
-c\bigl(1-\sqrt{1-|X|^2}\bigr).
\end{equation}
This is precisely the Fenchel dual inequality.  Moreover,
\begin{equation}\label{eq:second-term}
1-\sqrt{1-|X|^2} = \frac{|X|^2}{1+\sqrt{1-|X|^2}}
\le\frac{|X|^2}{1+\sqrt{1-s^2}}.
\end{equation} 
Together with
\eqref{eq:RTfirst}--\eqref{eq:RTsecond}, these imply
\begin{equation}\label{eq:semigroupcoercivity}
 \cQ_c(w)\ge\frac\lambda c
 \left(1-e^{-\lambda T}-\frac1{1+\sqrt{1-s^2}}\right)
 \operatorname{Var}_\nu(w).
\end{equation}
Indeed, substituting
$p=\nabla w$  and $X=(\lambda/c)\nabla R_Tw$  into
\eqref{eq:BVFenchel} yields
\[
\begin{aligned}
\mathcal Q_c(w)
&\ge \frac{\lambda}{c}
\int\nabla w\cdot\nabla R_Tw\,d\nu\\
&\quad-
\frac{\lambda^2}{c(1+\sqrt{1-s^2})}
\int|\nabla R_Tw|^2\,d\nu.
\end{aligned}
\]
Applying \eqref{eq:RTfirst} and
\eqref{eq:RTsecond} gives
\eqref{eq:semigroupcoercivity}.
Choose $T$ sufficiently large so that
$1-e^{-\lambda T}$ is close to $1$, and then
choose $r= r_0>0$ sufficiently small so that
$s<1$ and $s$ is close to zero whenever
$\operatorname{ess\,osc}  w\le r_0$.
The coefficient in \eqref{eq:semigroupcoercivity}
can thus be made arbitrarily close to
$\lambda/(2c)$, independently of $m$.

We give a direct BV justification which also records how the singular
variation is retained. 
For $w\in BV(\nu)$, write
$
D_\nu w=\nabla w\,\nu+D_\nu^s w,
$
where $D_\nu^s w$ is the singular part of the
weighted distributional derivative. 
For $w\in BV(\nu)$ satisfying
$\operatorname{ess\,osc} w\le r_0$
and $0<\epsilon<T$, set
\[
R_{\epsilon,T}w
=\int_\epsilon^T P_t(w-\nu w)\,dt,
\qquad
X_\epsilon=\frac{\lambda}{c}
\nabla R_{\epsilon,T}w.
\] 
The bounds above also apply to $X_\epsilon$, so
$\|X_\epsilon\|_\infty\le s<1$.  Thus,
the singular part satisfies
$
\int X_\epsilon\cdot dD_\nu^s w
\le |D_\nu^s w|(\mathbb R^m).
$ 
Since $\operatorname{div}_\nu\nabla=\mathcal L$, similarly to \eqref{eq:LR}, we have
$
\operatorname{div}_\nu X_\epsilon
=\frac{\lambda}{c}(P_Tw-P_\epsilon w).
$
Applying the  inequalities in \eqref{eq:BVFenchel} and \eqref{eq:second-term} to the absolutely continuous
part of $Dw$ and applying integration by parts yields
\begin{align*}
 \cQ_c(w)
 &\ge\int X_\epsilon\cdot dD_\nu w
       - c \int \frac{ |X_\epsilon|^2}{1+\sqrt{1-s^2}} \,d\nu \\
  &= \frac\lambda c\langle w,(P_\epsilon-P_T)w\rangle
        - c \int \frac{ |X_\epsilon|^2}{1+\sqrt{1-s^2}} \,d\nu.
\end{align*}
Since $P_\epsilon w\to w$ in $L^2(\nu)$,
\[
\langle w,(P_\epsilon-P_T)w\rangle
\longrightarrow
\langle g,(I-P_T)g\rangle
\ge (1-e^{-\lambda T})\operatorname{Var}_\nu(w).
\]
Since
$0\le e^{-\epsilon\xi}-e^{-T\xi}
\le1-e^{-T\xi}$ for $\xi\ge\lambda$,
the same spectral-calculus argument as in
\eqref{eq:RTsecond} gives
$
\int|\nabla R_{\epsilon,T}w|^2\,d\nu
\le\lambda^{-1}\operatorname{Var}_\nu(w).
$ 
Using these estimates,
we obtain \eqref{eq:semigroupcoercivity}
by letting $\epsilon\downarrow0$.
This proves the estimate directly for
bounded $BV$ functions, including their
singular variation.
\end{proof}

Recall the BV perimeter excess $\cE_I$ defined in \eqref{eq:E}. We now handle the cases of large deviations and singular BV derivatives. 

\begin{lemma}[Tail compensation estimate]
\label{lem:rare}
Let $u\in BV(\nu;[0,1])$, set
$a=\int u\,d\nu$, and suppose that
$
\int|u-a|\,d\nu\le hq
$
for some $h>0$ and $0<q\le1/2$. Then
\begin{equation}\label{eq:rare}
\cE_I(u)\ge L(q)R_h(u),
\end{equation}
where
\[
L(q)=\frac{H(q)}2-c_0,\qquad
c_0=I(1/2),\qquad
R_h(u)=\int(|u-a|-h)_+^2\,d\nu.
\]
\end{lemma}

\begin{proof}
Set
\[
r_+=(u-a-h)_+,\qquad
r_-=(a-h-u)_+,\qquad
r=r_++r_-.
\]
Their positivity sets are disjoint, and
Markov's inequality gives
\[
\nu\{r_+>0\}+\nu\{r_->0\}
=\nu\{|u-a|>h\}
\le\frac1h\int|u-a|\,d\nu
\le q.
\]
Moreover, $0\le r\le1$ and
$r^2=r_+^2+r_-^2$.

Define
\[
\varphi_+(t)=(t-a-h)_+^2,\qquad
\varphi_-(t)=(a-h-t)_+^2,
\]
so that $\varphi_\pm(u)=r_\pm^2$. For $t\in[0,1]$,
\[
|\varphi_+'(t)|+|\varphi_-'(t)|
=2\bigl((t-a-h)_++(a-h-t)_+\bigr)
\le2.
\]
The BV chain rule
\cite[Theorem~3.96]{AmbrosioFuscoPallara2000}
gives
\begin{equation} \label{eq:derivative}
\sum_\pm|\nabla(r_\pm^2)|
=2r|\nabla u|
\; \nu\text{-a.e.}
 \text{ and }
\sum_\pm|D^s(r_\pm^2)|_\nu
\le2|D^s u|_\nu.
\end{equation}
Indeed, on the Cantor part, the BV chain rule
\cite[Theorem~3.96]{AmbrosioFuscoPallara2000}
gives
\[
D^c(\varphi_\pm(u))
=\varphi_\pm'(\widetilde u)\,D^c u,
\]
where $\widetilde u$ denotes the approximate
limit of $u$, which exists
$|D^c u|$-almost everywhere. 
On the jump set $J_u$, let $u^+$ and $u^-$
denote the one-sided approximate traces.
Since $\varphi_+$ is nondecreasing,
$\varphi_-$ is nonincreasing, and $\varphi_\pm\in C^1$, the fundamental
theorem of calculus gives
\[
\begin{aligned}
&\sum_\pm
|\varphi_\pm(u^+)-\varphi_\pm(u^-)|\\
&\quad=
\int_{\min\{u^+,u^-\}}^{\max\{u^+,u^-\}}
\sum_\pm|\varphi_\pm'(t)|\,dt
\le2|u^+-u^-|.
\end{aligned}
\]
Thus, the second inequality in \eqref{eq:derivative} holds. 
By  \eqref{eq:derivative}, 
\begin{equation}\label{eq:rareBVchain}
\frac12\sum_\pm|D(r_\pm^2)|_\nu
\le\int r|\nabla u|\,d\nu+|D^s u|_\nu.
\end{equation}

Applying \eqref{eq:BVFenchel} pointwise with
$c=I(u)$, $p=\nabla u$, and
$X=r\nabla u/|\nabla u|$ (taken as zero
when $\nabla u=0$), and using
$I\le c_0$ and
$1-\sqrt{1-r^2}\le r^2$, we obtain
\[
\begin{aligned}
\cE_I(u)
&\ge\int r|\nabla u|\,d\nu
    +|D^s u|_\nu
    -c_0\int r^2\,d\nu\\
&\ge\frac12\sum_\pm|D(r_\pm^2)|_\nu
    -c_0\int(r_+^2+r_-^2)\,d\nu,
\end{aligned}
\]
where the last inequality follows from
\eqref{eq:rareBVchain}.

Finally, since $\nu\{r_\pm^2>0\}\le q$,
the small-support inequality
\eqref{eq:smallsupport} gives
$
|D(r_\pm^2)|_\nu
\ge H(q)\int r_\pm^2\,d\nu.
$
Consequently,
\[
\begin{aligned}
\cE_I(u) \ge\left(\frac{H(q)}2-c_0\right)
       \int(r_+^2+r_-^2)\,d\nu=L(q)R_h(u),
\end{aligned}
\]
as claimed.
\end{proof}
 
Lemma~\ref{lem:fiber} shows that $\cF_I(u)$ is a lower bound
for the perimeter of any set with fiber-volume function $u$.
We now establish local optimality by proving $\cF_I(u)\ge I(a)$
when the mean of $u$ is near a contact point $b$ and $u$ is close
to its mean in $L^1$.
\begin{proposition}[Uniform BV local optimality at a Gaussian contact]
\label{prop:local}
Assume $L=\infty$ and \eqref{eq:conditions}--\eqref{eq:contact}
at $b\in(0,1)$, and suppose $\mu$ is non-Gaussian. Put $\nu=\mu^{\otimes m}$. 
Then there exist
$\eta,\delta,\tau>0$, depending on $\mu,\kappa,b$ but not on $m$,
such that 
\begin{equation}\label{eq:local}
 \cF_I(u)-I(a)\ge\tau\cE_I(u)\ge0
\end{equation}
for every $u\in BV(\nu;[0,1])$ satisfying $a=\int u\dd\nu,\, |a-b|\le\eta$, and $\int|u-a|\dd\nu\le\delta.$ 
\end{proposition}
\begin{proof}
By Lemma~\ref{lem:spectralrigid},
$-I(b)I''(b)<K_{\mathrm{P}}$.
Choose $c_*>I(b)$ sufficiently close to $I(b)$ and then $K_*>0$ so that
\begin{equation}\label{eq:Kchoice}
 -\frac{I''(b)}2<K_*<\frac{K_{\mathrm{P}}}{2c_*}.
\end{equation}
Fix $\theta>0$ small and a compact interval $S\subset(0,1)$ containing $b$ such that
\begin{equation}\label{eq:contactlocalconstants}
 \max_SI<c_* \quad \text{ and } \quad
 \frac{A_*(1+\theta)}{K_*}<1,
\end{equation}
where $ A_*:=\frac12\max_S(-I'')$.
Select $h,\eta>0$ sufficiently small that
$2h\le r_0(K_{\mathrm{P}},c_*,K_*)$ from
Lemma~\ref{lem:spectralBV}, and
$[b-\eta-2h,b+\eta+2h]\subset S$.
For a function $u$ as in the statement,
\begin{equation}\label{eq:contactclip}
 w=\operatorname{clip}(u,[a-h,a+h]),\qquad
 R=\int(u-w)^2\dd\nu
 =\int(|u-a|-h)_+^2\dd\nu.
\end{equation}
Clipping decreases absolutely continuous gradients and the full
singular BV variation. 
By the BV chain rule,
$
\nabla w
=\mathbf1_{\{a-h<u<a+h\}}\nabla u, \, \nu\text{-a.e.}
$
Since $|a-b|\le\eta$ and
$[a-h,a+h]\subset S$, we have
$I(u)\le\max_S I<c_*$ wherever
$\nabla w\ne0$.
Since $c\mapsto\sqrt{c^2+t^2}-c$ decreases for $c>0$,
\begin{equation}\label{eq:contactenergy}
 \cE_I(u)\ge\cQ_{c_*}(w)\ge K_*\Var_\nu(w),
\end{equation}
where the second inequality follows by Lemma~\ref{lem:spectralBV}. 

Write $c_0=I(1/2)=\sup_{[0,1]}I$, and set
\[
 M_1=\max_{|a-b|\le\eta}|I'(a)|,\qquad
 C_0=\frac{c_0+M_1}{h^2}.
\]
For
$
D_a(s)=I(a)+I'(a)(s-a)-I(s)\ge0,
$
Taylor's theorem and the inclusion
$[a-2h,a+2h]\subset J$ give
$D_a(s)\le A_*(s-a)^2$ if $|s-a|\le2h.$
If $|s-a|>2h$, then $0\le I\le c_0$,
$|I'(a)|\le M_1$, and $|s-a|\le1$ yield
\[
D_a(s)\le c_0+M_1
=C_0h^2
\le C_0(|s-a|-h)^2.
\]
Combining the two cases, we obtain
\begin{equation}\label{eq:tangent}
D_a(s)\le A_*(s-a)^2
+C_0(|s-a|-h)_+^2,
\qquad 0\le s\le1.
\end{equation} 
Since $a=\int u\,d\nu$, the linear term
vanishes upon integration:
\[
\int D_a(u)\,d\nu
=I(a)-\int I(u)\,d\nu.
\]
This cancellation holds without assuming
$I'(a)=0$, allowing us to treat arbitrary
Gaussian contact points.
Thus for every $\theta>0$,
\begin{align}
 I(a)-\int I(u)\dd\nu
 &=\int D_a(u)\dd\nu\notag\\
 &\le A_*\Var_\nu(u)+C_0R\notag\\
 &\le A_*(1+\theta)\Var_\nu(w)
      +[C_0+A_*(1+\theta^{-1})]R.
 \label{eq:contactvar}
\end{align}
The last inequality follows from Young's inequality:
\[
\begin{aligned}
\Var_\nu(u)
&=\int\bigl[(w-\nu w)+(u-w-\nu (u-w))\bigr]^2\,d\nu\\
&\le(1+\theta)\Var_\nu(w)
 +(1+\theta^{-1})\Var_\nu(u-w)\\
&\le(1+\theta)\Var_\nu(w)
 +(1+\theta^{-1})R.
\end{aligned}
\]

If $\int|u-a|\dd\nu\le hq$ with $q\le1/2$,
Lemma~\ref{lem:rare} gives
$\cE_I(u)\ge L(q)R$ where
$L(q)=H(q)/2-c_0\to\infty$ as $q\downarrow0$.
Combining this with \eqref{eq:contactenergy} yields
\begin{equation}\label{eq:localclosure}
 I(a)-\int I(u)\dd\nu
 \le\left[
 \frac{A_*(1+\theta)}{K_*}
 +\frac{C_0+A_*(1+\theta^{-1})}{L(q)}
 \right]\cE_I(u).
\end{equation}
By \eqref{eq:contactlocalconstants}, choose $q$ so small that the
bracket is smaller than $1$, and put $\delta=hq$ and $\tau>0$ equal
to one minus that bracket. Since
$\cF_I(u)-I(a)=\cE_I(u)-[I(a)-\int I(u)]$, the claim follows. 
\end{proof}

\section{Dimension-free rigidity}
\label{sec:rigid}
We establish a dimension-free coordinate-localization
principle near arbitrary Gaussian contact points,
building on the Gaussian-transport and quantitative
stability approach of Glaudo \cite{Glaudo2024}.
The argument first locates low-perimeter competitors
near Gaussian half-spaces using the barycentric
stability theorem of Barchiesi, Brancolini, and Julin
\cite{BarchiesiBrancoliniJulin2017}.
We then develop a bounded test-field transfer inequality
that forces the normals of these half-spaces to
approach coordinate directions.
The resulting localization is qualitative for the
general theorem and quantitative for the
uniform-cube specialization. 

Set $\gamma_n$ for standard Gaussian probability measure, and write
$\dd\sigma_A=\phi_n\dd\mathcal H^{n-1}$ on the reduced boundary of a
finite-perimeter set $A\subset\R^n$. Let $T=F^{-1}\circ\Phi$. Let
\begin{equation}\label{eq:W}
 W(z,\mathbf n)=\sqrt{ \sum_{i=1}^n
      \left(\frac{s(z_i)}\kappa \mathbf n_i\right)^2},
 \qquad s=1/T'.
\end{equation}
For $E\subset\Omega^n$, let
$A=(T^{\otimes n})^{-1}(E)$, and write $\mathbf n_A$
for the measure-theoretic outward unit normal on $\partial^*A$.
Applying the area formula locally and then passing to the limit
over an increasing exhaustion of $\mathbb R^n$, we obtain
\begin{equation}\label{eq:area}
 P_{\mu^{\otimes n}}(E)=\kappa\int_{\partial^*A}W(z,\mathbf n_A)\dd\sigma_A.
\end{equation}
The inequality $s\ge\kappa$ implies $W\ge1$.
Define the nonnegative transport excess by
\begin{equation}\label{eq:R}
R:=\int_{\partial^*A}(W-1)\,d\sigma_A, \qquad W :=W(z,\mathbf n_A). 
\end{equation}

\begin{lemma}[Gaussian barycentric stability
{\cite{BarchiesiBrancoliniJulin2017}}]
\label{lem:}
Let $A\subset\mathbb R^n$ be a set of finite
Gaussian perimeter, with
$\gamma_n(A)=a\in(0,1)$. Set $z_a=\Phi^{-1}(a)$ 
and define its Gaussian barycenter, Gaussian perimeter deficit, and
barycentric deficit by
\[
b(A)=\int_A z\,d\gamma_n,\qquad \Delta=P_\gamma(A)-\IG(a), \qquad
\beta=\IG(a)-|b(A)|.
\]
Then
\begin{equation}\label{eq:barycenter}
\beta\le160\pi^3(1+z_a^2)\Delta.
\end{equation}
Furthermore, if $b(A)\ne0$, let
\[
v=-\frac{b(A)}{|b(A)|},\qquad
H=\{z:v\cdot z<z_a\},\qquad
\eta=\gamma_n(A\symdiff H).
\]
Then
\begin{equation}\label{eq:distance}
\eta^2\le4e^{-z_a^2/2}\beta.
\end{equation}
Moreover, the normal excess satisfies
\begin{equation}\label{eq:excess}
e:=\int_{\partial^*A}
|\mathbf n_A-v|^2\,d\sigma_A
=2(\Delta+\beta).
\end{equation}
\end{lemma}

\begin{proof}
Inequality \eqref{eq:barycenter} follows
from the Main Theorem of
\cite{BarchiesiBrancoliniJulin2017},
after converting to our normalization of
Gaussian perimeter. Inequality
\eqref{eq:distance} is Proposition~4
of the same reference.
Finally, Gaussian Gauss--Green gives
$
\int_{\partial^*A} \mathbf n_{A,i}\,d\sigma_A
=
-\int_A z_i\,d\gamma_n
$ for all $i$, 
implying
$
\int_{\partial^*A}\mathbf n_A\,d\sigma_A
=-b(A).
$
Since $|v|=1$ and $v=-b(A)/|b(A)|$,
\[
\begin{aligned}
e
&=2 P_\gamma(A)-2v\cdot
  \int_{\partial^*A}\mathbf n_A\,d\sigma_A =2(P_\gamma(A)-|b(A)|)
=2(\Delta+\beta).
\end{aligned}
\]
\end{proof}

We now combine the Gaussian stability estimates
in Lemma~\ref{lem:} with the transport excess
to localize the normal of the approximating
Gaussian half-space near a coordinate direction.
For a vector $w$, let
$
Q_{w}(z)
=\sum_i  w_{i}^2q(z_i).
$

\begin{lemma}[Test-field transfer and coordinate localization]
\label{lem:masterrigid}
In the above notation, suppose that an even,
nonnegative function $q\in C_b^1(\mathbb R)$
satisfies
\begin{equation}\label{eq:qassume}
 s(t)/\kappa-1\ge\rho q(t)
\qquad(t\in\mathbb R)
\end{equation}
for some $\rho>0$. Put
$
M=\|q\|_\infty$,  $ K=\|q'\|_\infty,
$
and
\[
\mathcal D_q
=\frac R\rho+2M\sqrt{e P_\gamma(A)}
 +(K+M|z_a|)\eta+M\beta.
\]
Then,  
\begin{equation}\label{eq:Qupper} 
\phi(z_a)\sum_i v_i^2
\E q(v_i z_a+\sqrt{1-v_i^2}\,Z) =\int_{\partial H}Q_v\,d\sigma_H \le\mathcal D_q, 
\end{equation}
where $Z\sim N(0,1)$.
Furthermore, with
$u=1-\max_i v_i^2$, there exists a
coordinate half-space $S\subset\Omega^n$
of the same measure as $E$ such that
\begin{equation}\label{eq:distance}
\mu^{\otimes n}(E\symdiff S)
\le\eta+\frac1\pi\arcsin\sqrt u
\le\eta+\frac12\sqrt u.
\end{equation}
\end{lemma}

\begin{proof} 
By Jensen's inequality and
\eqref{eq:qassume}, 
\[
\begin{aligned}
W :=W(z,\mathbf n_A)
&\ge \sum_{i=1}^n \mathbf n_{A,i}^2
\frac{s(z_i)}{\kappa} \ge \sum_{i=1}^n \mathbf n_{A,i}^2
\bigl(1+\rho q(z_i)\bigr) =1+\rho Q_{\mathbf n_A},
\end{aligned}
\] 
Moreover, 
\[
\begin{aligned}
|Q_v-Q_{\mathbf n_A}|
&\le M\sum_i|v_i^2-\mathbf n_{A,i}^2|\\
&\le M|v-\mathbf n_A|\,|v+\mathbf n_A| \le 2M|v-\mathbf n_A|.
\end{aligned}
\]
Integrating this inequality and using $R\ge\rho\int_{\partial^*A}Q_{\mathbf n_A}\,d\sigma_A$ and Cauchy--Schwarz, we obtain  
\begin{equation}\label{eq:QA}
\int_{\partial^*A}Q_v\,d\sigma_A
\le\frac R\rho+2M\sqrt{e P_\gamma(A)}.
\end{equation}
 
Since $A$ and $H$ have the same Gaussian
measure, $v=-b(A)/|b(A)|$, and
$b(H)=-\IG(a)v$, we have
\begin{align}
\beta
&=\IG(a)-|b(A)| \notag\\
&=\int_A(v\cdot z-z_a)\,d\gamma_n
  -\int_H(v\cdot z-z_a)\,d\gamma_n \notag\\
&=\int_{A\setminus H}(v\cdot z-z_a)\,d\gamma_n
  +\int_{H\setminus A}(z_a-v\cdot z)\,d\gamma_n \notag\\
&=\int_{A\triangle H}|v\cdot z-z_a|\,d\gamma_n.
\label{eq:alignedmoment}
\end{align}
Consequently,
\[
\int_{A\symdiff H}|v\cdot z|\,d\gamma_n
\le |z_a|\eta+\beta.
\]

Consider the bounded vector field
$X(z)=vQ_v(z)$. 
Since the outward
unit normal to $H$ is $v$, the Gaussian Gauss--Green
formula gives
\begin{align}
\int_{\partial H}Q_v\,d\sigma_H
&=\int_H\operatorname{div}_\gamma X\,d\gamma_n\\
&\le \int_A\operatorname{div}_\gamma X\,d\gamma_n
+\int_{A\symdiff H}
|\operatorname{div}_\gamma X|\,d\gamma_n\\
&=\int_{\partial^*A}
Q_v\langle v,\mathbf n_A\rangle\,d\sigma_A
+\int_{A\symdiff H}
|\operatorname{div}_\gamma X|\,d\gamma_n\\
&\le \int_{\partial^*A}Q_v\,d\sigma_A
+\int_{A\symdiff H}
|\operatorname{div}_\gamma X|\,d\gamma_n, \label{eq:integral-Q}
\end{align}
where the last inequality uses $Q_v\ge0$ and
$\langle v,\mathbf n_A\rangle\le1$.
Moreover,
\[
\operatorname{div}_\gamma X
=\sum_{i=1}^n v_i^3q'(z_i)-(v\cdot z)Q_v,
\]
so that
\[
|\operatorname{div}_\gamma X|
\le K+M|z_a|+M|v\cdot z-z_a|.
\]
Consequently, writing $\delta=\gamma_n(A\symdiff H)$,
\eqref{eq:QA}, \eqref{eq:alignedmoment}, and \eqref{eq:integral-Q} yield
\[
\int_{\partial H}Q_v\,d\sigma_H
\le \frac R\rho+2M\sqrt{e P_\gamma(A)}
+(K+M|z_a|)\delta+M\beta
=:\mathcal D_q.
\]

On the hyperplane $v\cdot z=z_a$,
the $i$th Gaussian coordinate has conditional
law $N(v_i z_a,1-v_i^2)$. This gives
the equality in \eqref{eq:Qupper}.

Finally, choose the signed coordinate
vector $\xi$ closest to $v$ and put
$\theta=\arccos(v\cdot\xi)=\arcsin\sqrt u$.
Rotational invariance of the Gaussian
measure in $\operatorname{span}\{v,\xi\}$
gives, for every $t\in\mathbb R$,
\[
\gamma_n\bigl(
\{v\cdot z<t\}\symdiff
\{\xi\cdot z<t\}\bigr)
\le\frac{\theta}{\pi}.
\] 
Taking $t=z_a$ and setting
$H_\xi=\{z:\xi\cdot z<z_a\}$, the triangle
inequality for symmetric differences gives
\[
\gamma_n(A\symdiff H_\xi)
\le \gamma_n(A\symdiff H)
 +\gamma_n(H\symdiff H_\xi)
\le \eta+\frac{\theta}{\pi}
\le \eta+\frac12\sqrt u.
\]
Since $S=T^{\otimes n}(H_\xi)$ is a
coordinate half-space of measure $a$,
the measure-preserving transport yields
\eqref{eq:distance}.
\end{proof}
 
 We now establish dimension-free rigidity near contact points, as promised. 

\begin{proposition}[Dimension-free rigidity near contact points]
\label{prop:rigid}
Assume $L=\infty$ and \eqref{eq:conditions}--\eqref{eq:contact}
at $b\in(0,1)$ and suppose $\mu$ is non-Gaussian.
If the dimensions $n_j$ are arbitrary,
$a_j\to b$, and
\begin{equation}\label{eq:rigidassump}
 \mu^{\otimes n_j}(E_j)=a_j,\qquad
 P_{ \mu^{\otimes n_j}}(E_j)\le I(a_j),
\end{equation}
then the measure distance of $E_j$ to the family of same-volume
coordinate half-spaces tends to zero. 
Consequently, there is $\eta_\delta>0$,
independent of dimension, such that
\begin{equation}\label{eq:rigid}
 |a-b|\le\eta_\delta,\quad P_{\mu^{\otimes n}}(E)\le I(a),
 \quad\mu^{\otimes n}(E)=a
 \quad\Longrightarrow\quad
 \mu^{\otimes n} (E\symdiff S)\le\delta
\end{equation}
for some same-volume coordinate half-space $S$.
\end{proposition}
\begin{proof}
The monotone Gaussian transport satisfies $s(z)=1/T'(z)\ge\kappa$.
Since $\mu$ is non-Gaussian, $s/\kappa-1$ is a continuous, even,
nonnegative function that is not identically zero. There are
$\rho>0$ and a nonzero, even, nonnegative
$q\in C_c^\infty(\mathbb R)$ such that
\begin{equation}\label{eq:qqual}
 \frac{s(z)}\kappa-1\ge\rho q(z)\qquad(z\in\mathbb R).
\end{equation}
Let $A_j=(T^{\otimes n_j})^{-1}(E_j)$ and use the notation
$\Delta_j,R_j,\beta_j,\eta_j,e_j,v_j$ of
Lemma~\ref{lem:masterrigid}. 
By Gaussian isoperimetry and $W\ge1$, we have
$\Delta_j,R_j\ge0$. Moreover, \eqref{eq:area}
and \eqref{eq:rigidassump} yield
\begin{equation}\label{eq:rigidvanish}
0\le\Delta_j+R_j
=\frac{P_{\mu^{\otimes n_j}}(E_j)}{\kappa}-\IG(a_j)
\le\frac{I(a_j)}{\kappa}-\IG(a_j)
\longrightarrow0,
\end{equation}
where the limit follows from the contact condition
$I(b)=\kappa\IG(b)$ and continuity. 
Put $t_j=\Phi^{-1}(a_j)\to t_b=\Phi^{-1}(b)$. 
By \eqref{eq:barycenter}--\eqref{eq:excess}
and \eqref{eq:rigidvanish}, we have
$\beta_j,\eta_j,e_j\to0$. Moreover,
$p_j=P_\gamma(A_j)= \IG(a_j)+\Delta_j$ is bounded and
$|b(A_j)|=\IG(a_j)-\beta_j\to\IG(b)>0$.
Thus the vector $v_j$ is eventually well-defined, and
\eqref{eq:Qupper} gives
\begin{equation}\label{eq:qualQsmall}
0\le \int_{\partial H_j}Q_{v_j}\,d\sigma_{H_j}
\le \frac{R_j}{\rho}+2M\sqrt{e_jp_j}
+(K+M|t_j|)\eta_j+M\beta_j
\longrightarrow0.
\end{equation}

Suppose along a subsequence $\max_i v_{j,i}^2\le1-d$ for
some $d>0$. Choose $T_0>|t_b|+1$. On $H_j$ the $i$th Gaussian
coordinate has conditional distribution
$N(v_{j,i}t_j,1-v_{j,i}^2)$, with variance in $[d,1]$
and mean in $[-T_0,T_0]$ for all large $j$.
Because $q\ge0$ is continuous and nonzero, strict positivity
of Gaussian densities and compactness imply
\begin{equation}\label{eq:compactpositive}
 m_{d,T_0}:=
 \min_{\substack{|r|\le T_0\\\sqrt d\le\sigma\le1}}
 \mathbb E q(r+\sigma Z)>0.
\end{equation}
The exact conditional Gaussian formula therefore yields
\[
\int_{\partial H_j}Q_{v_j}\,d\sigma_{H_j}=  \int_{\partial H_j}\sum_i v_{j,i}^2q(z_i)\,d\sigma_{H_j}
 \ge\phi(t_j)m_{d,T_0},
\]
contrary to \eqref{eq:qualQsmall}, because $\phi(t_j)\to\phi(t_b)>0$.
Hence $1-\max_i v_{j,i}^2\to0$.
Since $u_j:=1-\max_i v_{j,i}^2\to0$ and
$\eta_j\to0$, \eqref{eq:distance} yields
same-volume coordinate half-spaces $S_j$ such that
\[
\mu^{\otimes n_j}(E_j\symdiff S_j)
\le\eta_j+\frac12\sqrt{u_j}\longrightarrow0.
\]

If there is no such $\eta_\delta>0$ described in the proposition,  choosing $|a_j-b|\to0$ would contradict the sequential
rigidity given by the first statement. 
\end{proof}

\begin{remark}[Exact Gaussian normalization for the cube]
In the quantitative specialization it is important to keep the
normalization of every perimeter term. For a cube set $E\subset Q_n$
and its Gaussian pullback $A=(\Phi^{\otimes n})^{-1}(E)$,
the cofactor area formula gives
\begin{align*}
 P(E;Q_n)
 &=\int_{\partial^*A}|\operatorname{cof}D(\Phi^{\otimes n})
          \,\mathbf n_A|\,d\mathcal H^{n-1}\\
 &=\sqrt{2\pi}\int_{\partial^*A}
  \sqrt{\sum_{i=1}^n\mathbf n_{A,i}^2e^{z_i^2}}
        d\sigma_A.
\end{align*}
Indeed $D(\Phi^{\otimes n})$ is diagonal with entries $\phi(z_i)$,
so its cofactor acting on a unit normal has Euclidean length
$\phi_n(z) \sqrt{ \sum_i\mathbf n_i^2/\phi(z_i)^2}$.
For global finite-perimeter sets, first apply this on the open
sets $\Phi^{\otimes n}(B_R)$, then use monotone continuity of
relative perimeter as $R\to\infty$. No artificial boundary on
$\partial B_R$ is added. Consequently the exact Gaussian perimeter
and transport excess decomposition reads
\[
 \frac{P(E;Q_n)}{\sqrt{2\pi}}
 =P_\gamma(A)+R,
 \qquad R=\int_{\partial^*A}(W-1)\,d\sigma_A\ge0,
\]
where $W=\sqrt{\sum_i\mathbf n_i^2e^{z_i^2}}$.
Thus the cube perimeter comparison, including its nonnegative
transport excess, is a specialization of the general area formula.
\end{remark}

\section{Proofs of the main results}
\label{sec:proofs}
The local-optimality and rigidity estimates first close the
isoperimetric comparison for full-support measures. The finite-support
case then follows from the profile-preserving comparison model, while
the quantitative appendices remain independent of that reduction.
The qualitative Gaussian-contact result uses constants depending on
the measure and its contact point. The stronger central constants for the cube and  circle cases
 are proved in Appendix~\ref{app:circlequant}. Equality cases follow from fiber reduction and
one-factor rigidity.

\begin{proof}[Proof of Theorem~\ref{thm:general}]
We first prove the result for $L=\infty$.
If $\mu$ is Gaussian, then $I=\kappa\IG$ by the contact condition,
and Gaussian isoperimetry proves the assertion for all volumes.
Suppose $\mu$ is non-Gaussian. Choose $\eta,\delta,\tau$
from Proposition~\ref{prop:local} for the contact point $b$. 
By Proposition~\ref{prop:rigid}, there is $\eta_\delta>0$,
independent of dimension, such that \eqref{eq:rigid} holds. 
Put
$\varepsilon_{\mu,b}=\min\{\eta,\eta_\delta\}$, decreasing it if
necessary so both contact intervals remain inside $(0,1)$.

The case $n=1$ is immediate. If $n\ge2$, we need only consider
competitors with $P_{\mu^{\otimes n}}(E)\le I(a)$, since a single-coordinate
half-space yields the reverse upper bound. By permuting and reflecting
the coordinates, write $S=\{x_1<F^{-1}(a)\}$ and let
$u(y)=\mu(E_y)$, where $y=(x_2,\ldots,x_n)$.
Then
\[
 \int u\dd\mu^{\otimes (n-1)}=a,\qquad
 \int|u-a|\dd\mu^{\otimes (n-1)}
 \le \int|\mathbf 1_E-\mathbf 1_S|\,d\mu^{\otimes n} 
 \le\mu^{\otimes n}(E\symdiff S)\le\delta.
\]
The full BV fiber inequality in Lemma \ref{lem:fiber} and local optimality in Proposition~\ref{prop:local} give
\begin{equation}\label{eq:closing}
 P_{\mu^{\otimes n}}(E)\ge\cF_I(u)
 \ge I(a)+\tau\cE_I(u)\ge I(a).
\end{equation}
This proves $I_n(a)=I(a)$ on $|a-b|\le\varepsilon_{\mu,b}$.
Since $I_n(a)=I_n(1-a)$ and $I(a)=I(1-a)$ by complements,
the same radius works about $1-b$.

If equality holds and $\mu$ is non-Gaussian, (\ref{eq:closing})
forces $\cE_I(u)=0$, so the absolutely continuous gradient and
singular variation of $u$ vanish; connectedness implies $u\equiv a$.
Furthermore,
\[
 P_{\mu^{\otimes n}}(E)
 \ge\int_{\partial^*E}|\mathbf n_1|\dd P_E
 =\int P_\mu(E_y)\dd\mu^{\otimes (n-1)}(y)
 \ge I(a).
\]
All inequalities are equalities. Thus $|\mathbf n_1|=1$ almost
everywhere on the reduced boundary and the transverse derivatives
of $\one_E$ vanish. Consequently $E=A\times\Omega^{n-1}$
modulo null sets, with $\mu(A)=a$ and $P_\mu(A)=I(a)$.
This proves the classification for $L=\infty$.

Now suppose $0<L<\infty$. Apply
Lemma~\ref{lem:fullsupportminorant} to obtain a full-support
non-Gaussian measure $\widehat\mu$ with profile $J$, contact
radius $\zeta>0$, and a nondecreasing $1$-Lipschitz quantile
transport $\Theta:\mathbb R\to(-L,L)$ sending
$\widehat\mu$ to $\mu$. The already proved full-support theorem
gives $\varepsilon_0>0$ such that
\[
 I_{\widehat\mu^{\otimes n}}(a)=J(a)
 \qquad(n\ge1,\quad |a-b|\le\varepsilon_0
   \text{ or }|a-(1-b)|\le\varepsilon_0),
\]
and all minimizers there are coordinate cylinders. Decrease the
radius to $\varepsilon_{\mu,b}<\min\{\varepsilon_0,\zeta\}$.

For any finite-perimeter set $E\subset(-L,L)^n$ of mass $a$ in
these intervals, put $A=(\Theta^{\otimes n})^{-1}(E)$ and
write $\widehat f_n$ for the density of $\widehat\mu^{\otimes n}$.
The coordinatewise area formula, first on compact subsets and
then by exhaustion, gives
\begin{equation}\label{eq:finitearea}
\begin{aligned}
 P_{\mu^{\otimes n}}(E)
 &=\int_{\partial^*A}
\sqrt{ \sum_{i=1}^n\frac{\mathbf n_{A,i}^2}
  {\Theta'(z_i)^2}}
  \widehat f_n(z)\,d\mathcal H^{n-1}(z)\\
 &\ge P_{\widehat\mu^{\otimes n}}(A)
 \ge J(a)=I(a),
\end{aligned}
\end{equation}
where the first inequality follows since  $0<\Theta'(x) \le1.$
The formula remains valid in the extended-real sense, even if an intermediate perimeter $ P_{\widehat\mu^{\otimes n}}(A)$ is infinite.
A coordinate half-interval gives the opposite inequality,
so $I_{\mu^{\otimes n}}(a)=I(a)$ for every $n$.
Equivalently, the profile comparison in \eqref{eq:finitearea}
may be viewed as a special case of Ros's comparison theorem
(Lemma~\ref{lem:roscomparison}).

If $E$ is an isoperimetric minimizer, then
$P_{\mu^{\otimes n}}(E)=I(a)$, so equality holds
throughout \eqref{eq:finitearea}. Hence
$P_{\widehat\mu^{\otimes n}}(A)=J(a)$, and the
full-support classification implies that $A$
is a coordinate cylinder. Since $\Theta$ is
a monotone bijection, $E$ is also a coordinate
cylinder, over a one-dimensional minimizer of $\mu$.
The profile construction is in the spirit of the approximation
principles in \cite[Section~6]{Milman2009}, but it preserves the
exact Gaussian contact instead of passing a contact interval to a limit.
\end{proof}
\begin{remark}
For $L<\infty$, an alternative approach would be to establish a local-optimality lemma using the Neumann diffusion semigroup and then follow the same steps as in the full-support case $L=\infty$. 
\end{remark}

\begin{proof}[Proof of Corollary~\ref{cor:sqrt}]
Let $U=V-V(0)$. 
Convexity and nonnegativity of $\sqrt U$ imply convexity of $U$,
so \eqref{eq:concavityequivalence} gives the required concavity of $I$.
Let
\[
 h(y)=c e^{-\kappa^2y^2/2},\quad
 G(y)=\Phi(\kappa y),\quad
 r(x)=\frac{\sqrt2}{\kappa}\sqrt{U(x)}\quad(0\le x<L).
\]
The Gaussian $h$ is normalized, $r$ is convex, and $h(r(x))=f(x)$.
For $D(x)=F(x)-G(r(x))$ one has, wherever derivatives exist,
\[
 D'(x)=f(x)(1-r'(x)),\quad D(0)=0,\quad
 \lim_{x\uparrow L}D(x)=1-\lim_{x\uparrow L}G(r(x))\ge0.
\]
Since $r'$ is nondecreasing, $D'$ changes sign at most once from
nonnegative to nonpositive, and therefore $D\ge0$.
For $a=F(x)\ge1/2$, if $y=G^{-1}(a)$, then $y\ge r(x)$ and hence
$I(a)=h(r(x))\ge h(y)=\kappa \IG(a)$. Symmetry yields the comparison
on the full unit interval.

Theorem~\ref{thm:general} therefore applies without any separate
curvature verification. If the measure is Gaussian, the last assertion
of that theorem gives the full-volume statement and the corresponding
classification of all minimizing half-spaces.
\end{proof}

\begin{remark}[Concrete potentials and critical curvature]
The square-root class includes
$U(x)=\alpha x^2+\beta|x|^p$, $\alpha\ge0$, $\beta>0$, $p>2$,
as well as finite sums $\sum_j a_j|x|^{p_j}$ with $a_j\ge0$ and
$p_j\ge2$ whenever the resulting measure is integrable.
Indeed, $\sqrt U$ is the Euclidean norm of the nonnegative convex
functions $\sqrt{a_j}|x|^{p_j/2}$. It also includes
$U(x)=\cosh x-1$, since
$\sqrt{U(x)}=\sqrt2\sinh(x/2)$ for $x\ge0$, and
$U\equiv0$ on a finite symmetric interval.
The centered Gaussian comparison implies $k=U''(0)\le\kappa_{\mathrm c}^2$;
the strengthened theorem also covers non-Gaussian measures with
$k=\kappa^2$. For them, the proof uses the strict spectral gap
$K_{\mathrm{P}}>\kappa^2$ and a nonzero (not necessarily quadratic)
transport witness. Within the square-root class the equality
$k=\kappa_{\mathrm c}^2$ forces an untruncated Gaussian, but this rigidity
is not true for the full comparison class. Neither theorem claims
a maximal central interval width or a converse characterization
of all log-concave product measures with exact central isoperimetry.
\end{remark}

The explicit interval in Theorem~\ref{thm:cube} is established by
an endpoint-atom local BV inequality and a quadratic-core Gaussian test
field. The complete proof, including the finite interval-arithmetic
certificate, is given in Appendix~\ref{app:circlequant}.

\begin{proof}[Proof of Theorem~\ref{thm:cube}]
Appendix~\ref{app:circlequant} proves the $9/625$ plateau for every
unit cube using a dimension-free local BV estimate with
$\Var(u)\le0.01942$ and a Gaussian localization bound
$\Var(u)<0.019403$. It also proves the equality classification.
The same appendix establishes the exact all-volume normalized
profile identity between an $n$-fold cube and an $n$-fold product of
circles of circumference two, by successive BV fiber
rearrangements in one direction and reflection/doubling in the other.
A circle of radius $R$ is a metric dilation of that reference circle
by $\pi R$, so perimeter is multiplied by $(\pi R)^{-1}$.
The equality case is traced back through the rearrangements, giving
arc cylinders and proving both assertions.
\end{proof} 

We next transfer the proved one-dimensional theorem to manifolds.
This uses only the comparison lemmas of Section~\ref{sec:tools};
no second local-optimality or rigidity argument on $M$ is required.

\begin{proof}[Proof of Theorem~\ref{thm:manifold}]
Apply Lemma~\ref{lem:modelprofile} with $J=I$.
For the resulting measure $\mathfrak m=\mathfrak m_J$, one has
$I_{\mathfrak m}=I$ and $f_J(0)=I(1/2)$.
Assumptions \eqref{eq:manifoldcomparison}--\eqref{eq:manifoldcontact}
are exactly the Gaussian comparison and contact required in
Theorem~\ref{thm:general} for the model measure.
It follows that, on the two intervals about $b$ and $1-b$
of a common positive radius depending only on $I,b$,
\[
I_{\mathfrak m^{\otimes n}}(a)=I(a)
\qquad(n\ge1).
\]
Repeated application of Lemma~\ref{lem:roscomparison} yields
\begin{equation}\label{eq:manifoldchain}
\begin{aligned}
I_n(a)=I_{\mu^{\otimes n}}(a)
&\ge I_{\mathfrak m\otimes\mu^{\otimes(n-1)}}(a) \ge\cdots\ge
I_{\mathfrak m^{\otimes n}}(a)=I(a).
\end{aligned}
\end{equation}
Conversely, taking a cylinder over a single-factor
minimizing sequence gives $I_n(a)\le I(a)$.
This proves \eqref{eq:manifoldcentral}.
If $I$ is a scaled Gaussian profile, then
$\mathfrak m$ is Gaussian and the same argument
works at every volume. We make no equality
classification in that case.

Assume that $I$ is not a scaled Gaussian profile,
and let $E\subset M^n$ be an isoperimetric minimizer.
Successively rearrange each factor into the
one-dimensional model $\mathfrak m$, obtaining
\[
E^{(0)}=E,\ E^{(1)},\ldots,E^{(n)},
\qquad
E^{(r)}\subset(-\ell_J,\ell_J)^r\times M^{n-r}.
\]
Each step preserves volume and does not increase
perimeter. Since the initial and final perimeters
both equal $I(a)$, equality holds at every step.

The successive rearrangements preserve monotonicity
in previously rearranged coordinates. Hence $E^{(n)}$
is coordinatewise downward monotone, and the equality
classification in Theorem~\ref{thm:general} yields
\begin{equation}\label{eq:finalmodelcylinder}
E^{(n)}
=\{t\in(-\ell_J,\ell_J)^n:t_j<Q_J(a)\}
\end{equation}
for some $j$, modulo null sets, where $Q_J$
is defined in Lemma~\ref{lem:modelprofile}.

Reversing the rearrangements for $r>j$ preserves
this cylinder, since its fibers in those directions
are either empty or full. At step $j$, let
$A=E^{(j-1)}\subset M\times Y$, where
\[
Y=(-\ell_J,\ell_J)^{j-1}\times M^{n-j},
\qquad
\nu_Y=\mathfrak m^{\otimes(j-1)}
\otimes\mu^{\otimes(n-j)}.
\]
Since the $j$th rearrangement preserves fiber masses,
we have $\mu(A_y)=a$ for $\nu_Y$-almost every $y$.
Decompose the outward unit normal as
$\mathbf n_A=(\mathbf n_M,\mathbf n_Y) 
$ with $|\mathbf n_M|^2+|\mathbf n_Y|^2=1.
$
By the BV slicing formula and the definition of $I$,
\[
\begin{aligned}
I(a)
&=P_{\mu\otimes\nu_Y}(A) \ge\int_Y P_\mu(A_y)\,d\nu_Y(y)\\
&=\int_{\partial^*A}|\mathbf n_M|\,d\sigma_A \ge I(a).
\end{aligned}
\]
Thus equality holds throughout, forcing
$|\mathbf n_M|=1$ and $\mathbf n_Y=0$
almost everywhere on the weighted reduced boundary.
Consequently, the distributional derivative of
$\one_A$ in the $Y$ directions vanishes.
By connectedness of $Y$,
$A=A_0\times Y$ modulo null sets, where
$\mu(A_0)=a$ and $P_\mu(A_0)=I(a)$.

The remaining reverse rearrangements preserve
this cylinder, since all corresponding fibers
have mass zero or one. Thus $E$ is a cylinder
over a one-factor isoperimetric minimizer.
\end{proof}
 
\begin{proof}[Proof of Corollary~\ref{cor:sphere}]
Barthe \cite[Sections~2.2 and 2.3]{Barthe2001}
established the sharp Gaussian-profile comparisons
for round spheres and Euclidean balls, with contact
at $a=1/2$. Their one-factor isoperimetric profiles
are concave and $C^2$ on $(0,1)$, as follows from
the classical descriptions of their minimizers by
spherical caps and orthogonal-ball regions,
respectively (with the one-dimensional cases being
immediate). Thus Theorem~\ref{thm:manifold} applies,
yielding the asserted dimension-free intervals and
the classification of minimizers.

Scaling by $R$ multiplies normalized perimeters by
$R^{-1}$ while preserving relative volumes, so the
interval widths are independent of $R$.
The explicit one-dimensional bounds follow from
the cube theorem and the circle--cube comparison.
For $d=2$, the spherical cap formula gives
$
I_{S_R^2}(a)=R^{-1}\sqrt{a(1-a)}.
$
\end{proof}

\begin{remark}[A smooth lower model with Gaussian contact]
\label{rem:lowermodel}
Global $C^2$ regularity of the actual manifold profile is not essential
if an appropriate lower model is available. Suppose that a positive,
symmetric, concave $J\in C^2((0,1))$, with $J(0)=J(1)=0$, satisfies
\[
 I(a)\ge J(a)\ge\kappa\IG(a)
 \quad(0<a<1),\qquad
 J(b)=\kappa\IG(b),
\]
and
\[
 I(a)=J(a)
 \quad\text{when }\min\{|a-b|,|a-(1-b)|\}\le\eta
\]
for some $\eta>0$. Construct $\mathfrak m$ by
Lemma~\ref{lem:modelprofile} and restrict its contact intervals to
radius $\min\{\eta,\varepsilon_{\mathfrak m,b}\}$. The comparison chain
and the cylinder upper bound give
$I_{n}(a)=I(a)$ there, and the same equality
classification applies when $J$ is not a scaled Gaussian profile.
This permits phase transitions and nonsmoothness away from the
contact intervals; an arbitrary global lower bound without equality
near the contact does not provide exact isoperimetry.
\end{remark}

\begin{remark}[Geometric versus profile hypotheses]
\label{rem:manifoldscope}
In Theorem~\ref{thm:manifold}, symmetry is automatic for the profile
by taking complements, and no reflection isometry of $M$ is assumed.
A nonnegative Bakry--\'Emery Ricci tensor, with convex boundary in
the boundary case, provides profile concavity under the standard
smooth hypotheses; see \cite[Proposition~2.3]{Milman2026}. It does
not by itself supply the normalized comparison
\eqref{eq:manifoldcomparison} or global $C^2$ profile regularity.
The theorem is consequently stated in terms of the actual profile,
not as an assertion about all manifolds with nonnegative curvature.
The Gaussian-profile exception also concerns the comparison model,
not an identification of the original manifold with Gaussian space.
\end{remark}
 
\appendix
\section{Quantitative cube and circle isoperimetry}\label{app:circlequant}
We prove the explicit cube and circle intervals, including the equality
cases. The Gaussian transport comparison and stability argument follow
\cite[Proposition~3.2 and Section~5]{Glaudo2024}; the local BV estimate,
bounded test-field transfer, and certified constants below are quantitative
refinements not supplied by that reference. The arithmetic certificate treats
all decimal inputs as exact rationals. We use the author-posted
\emph{improved version} of \cite{BarchiesiBrancoliniJulin2017}
(\url{https://cvgmt.sns.it/media/doc/paper/2516/gauss-improved.pdf})
for its sharper explicit constants.

\subsection{Circles and cubes have the same normalized profile}

Let $C_2=\mathbb R/(2\mathbb Z)$ be equipped with
normalized length measure. By the reflection/doubling
argument of Milman \cite[Remark~2.4]{Milman2026},
\begin{equation}\label{cc:eq:profileidentity}
I_{C_2^n}(a)=I_{Q_n}(a),
\qquad 0\le a\le1.
\end{equation}
Since $S_R^1$ is a metric dilation of $C_2$
by $\pi R$, it follows that
\begin{equation}\label{cc:eq:scaledidentity}
I_{(S_R^1)^n}(a)=\frac1{\pi R}I_{Q_n}(a).
\end{equation}
The equality classification for circles is recovered from the cube
classification at the end of this appendix by tracking equality in
fiber rearrangement.

\subsection{A stronger local BV estimate}

Let $\nu$ be uniform probability on $Q_m$. For $u\in BV(Q_m;[0,1])$, put
\begin{align}
 \cG(u)&=\int\bigl(\sqrt{1+|\nabla u|^2}-1\bigr)\,d\nu+|D^su|_\nu,\label{cc:eq:G}\\
 s(u)&=\nu\{u=0\}+\nu\{u=1\}.
\end{align}
The BV fiber functional for the uniform interval (cf.
Lemma~\ref{lem:fiber}) is
\[
\cF(u):=\int\sqrt{\one_{(0,1)}(u)^2+|\nabla u|^2}\,d\nu
       +|D^su|_\nu.
\]
By BV locality, $\nabla u=0$ almost everywhere on $\{u=0,1\}$;
thus
\begin{equation}\label{cc:eq:F}
\cF(u)=1-s(u)+\cG(u).
\end{equation}

Write $\phi(z)=(2\pi)^{-1/2}e^{-z^2/2}$,
$\Phi(z)=\int_{-\infty}^z\phi(t) \dd t$, and $\phi_0=\phi(0)$.
We use the established Gaussian comparison
\cite[Proposition~3.2]{Glaudo2024}
\begin{equation}\label{cc:eq:gaussiancomparison}
P(B;Q_m)\ge\sqrt{2\pi}\,\IG(\nu(B)).
\end{equation}
We use the decreasing function
\begin{equation}\label{cc:eq:H}
 H(q)=\frac{\sqrt{2\pi}\,\IG(q)}q
     =\frac{\exp\{-\Phi^{-1}(q)^2/2\}}q.
\end{equation}
Its monotonicity follows from concavity of $\IG$ and $\IG(0)=0$.

\begin{proposition}[Local optimality with endpoint atoms]\label{cc:prop:local}
Suppose
\begin{equation}\label{cc:eq:localhyp}
 a=\nu u,\qquad |a-1/2|\le\epsilon:=0.0144,\qquad
 \Var_\nu(u)\le V:=0.01942.
\end{equation}
For $w=\clip(u,[0.301,0.699])$ one has
\begin{equation}\label{cc:eq:localclaim}
 \cF(u)\ge1+\frac1{100}\Var_\nu(w).
\end{equation}
If $\cF(u)\le1$, then $u=a$ almost everywhere.
\end{proposition}
We next prove this proposition. 

\subsubsection{An explicit BV perimeter–variance estimate for cubes}

The following lemma is an explicit cube analogue
of the small-oscillation BV perimeter–variance estimate
in Lemma~\ref{lem:spectralBV}. Although both
rely on semigroup estimates and Fenchel duality,
the latter does not directly apply to the
uniform measure on a bounded interval.
Here we lift the cube to a product of spheres
with positive curvature, obtaining an explicit
coercivity coefficient for the finite oscillation
range required in our quantitative argument.

\begin{lemma}\label{cc:lem:core}
For $w\in BV(Q_m)$ with essential oscillation at most $2h$,
\begin{equation}\label{cc:eq:core}
 \cG(w)\ge K(h)\Var(w),\qquad
 K(h)=\frac8{1+\sqrt{1+8\pi h^2}}.
\end{equation}
\end{lemma}
\begin{proof}
Let $M=(S_{1/2}^2)^m$ with normalized area measure, and let $\Pi:M\to[0,1]^m$ map each factor to its third Euclidean coordinate plus $1/2$. This map pushes normalized measure to uniform measure and is $1$-Lipschitz: its component gradients have norm at most one and act on orthogonal factors. For smooth $w$, $W=w\circ\Pi$ satisfies
\[
 \Var_M(W)=\Var_{Q_m}(w),\qquad \cG_M(W)\le\cG_{Q_m}(w).
\]

On $M$, the heat semigroup has curvature lower bound $4$ and spectral gap $8$. For the latitude functions used here the latter can also be checked from the Legendre eigenvalues $4k(k+1)$, and it tensorizes. The reverse isoperimetric inequality \cite[Proposition~2.1]{BakryGentilLedoux2015} implies, for $0\le f\le1$,
\[
 |\nabla P_t f|\le\frac{2\phi_0}{\sqrt{e^{8t}-1}}.
\]
Let $RW=\int_0^\infty P_t(W-\nu W)\,dt$. 
Since $\operatorname{ess\,osc}(W)\le 2h$,
the reverse isoperimetric gradient estimate gives
\[
\|\nabla P_tW\|_\infty
\le\frac{4h\phi(0)}{\sqrt{e^{8t}-1}}.
\]
Consequently, using
$\int_0^\infty(e^{8t}-1)^{-1/2}\,dt=\pi/8$,
\begin{align}
\|\nabla RW\|_\infty
\le 4h\phi(0)\frac{\pi}{8}
=h\sqrt{\frac{\pi}{8}}=:M_h. \label{cc:eq:resolvent}
\end{align}
Moreover, writing $g=W-\nu W$ and using
$-LRW=g$, integration by parts and the
spectral-gap contraction yield
\begin{align}
\int|\nabla RW|^2\,d\nu
&=\langle RW,g\rangle_\nu =\int_0^\infty\langle P_tg,g\rangle_\nu\,dt \nonumber \\
&\le\int_0^\infty e^{-8t}\,dt\,
\operatorname{Var}_\nu(W)
=\frac18\operatorname{Var}_\nu(W). \label{cc:eq:resolvent2}
\end{align}

Arguing as in the proof of
\eqref{eq:semigroupcoercivity}, Fenchel duality
with $X=k\nabla RW$, where $kM_h<1$,
and the resolvent estimates
\eqref{cc:eq:resolvent} and \eqref{cc:eq:resolvent2} yield
\[
\begin{aligned}
\cG_M(W)
&\ge k\int\langle\nabla W,\nabla RW\rangle\,d\nu -\frac{k^2}{1+\sqrt{1-k^2M_h^2}}
\int|\nabla RW|^2\,d\nu \\
&\ge
\left[k-\frac{k^2}
{8(1+\sqrt{1-k^2M_h^2})}\right]\Var_\nu(W).
\end{aligned}
\]
Unlike Lemma~\ref{lem:spectralBV}, where the
resolvent must be truncated to control its
$L^\infty$ gradient, the positive curvature
of $M$ yields an integrable semigroup gradient
bound, allowing us to take $T=\infty$ directly.

Taking $k=8/\sqrt{1+64M_h^2}$ gives \eqref{cc:eq:core}. For the BV version on $M$, replace $R$ by $R_\delta=\int_\delta^\infty P_t\,dt$ applied to $W-\nu W$. The bounded smooth dual field pays for the full singular variation, and integration by parts gives $k\langle W-\nu W,P_\delta(W-\nu W)\rangle$, tending to $k\Var(W)$. The bounds \eqref{cc:eq:resolvent} remain valid. On the cube, first apply the smooth inequality and then use bounded area-strict approximants, preserving the essential range. This proves the full relative BV assertion without assigning any perimeter to the cube boundary.
\end{proof}

\subsubsection{Clipping and the tail angle}
Set
\begin{equation}\label{cc:eq:localconstants}
 h=0.199,\quad \ell=1/2-h=0.301,\quad c=1/2-\epsilon=0.4856,\quad b_0=0.037748.
\end{equation}
Let $b=\nu w$. For a mean-zero random variable $X$ of variance at most $V$ and $g>0$,
\begin{equation}\label{cc:eq:stoplossmean}
 \E(X-g)_+=\tfrac12(\E|X-g|-g)
 \le\tfrac12(\sqrt{V+g^2}-g).
\end{equation}
The identity follows from $x_+=(|x|+x)/2$
and $\E X=0$, while the inequality follows
from Jensen's inequality. 
Since $a=\nu u$, $b=\nu w$, $\ell=1/2-h$, and 
$w=u+(\ell-u)_+-(u-(1-\ell))_+$, 
we have
\[
b-a
=\underbrace{\int(\ell-u)_+\,d\nu}_{C_-}
-\underbrace{\int(u-(1-\ell))_+\,d\nu}_{C_+}.
\]
Since $C_-,C_+\ge0$, we have
$|b-a|\le\max\{C_-,C_+\}$.
Moreover, $|a-1/2|\le\epsilon$ implies
$a-\ell,\,1-\ell-a\ge h-\epsilon$.
Applying \eqref{cc:eq:stoplossmean} to
$X=a-u$ and $X=u-a$, respectively, yields
\begin{equation}\label{cc:eq:meanclip}
\begin{aligned}
|b-1/2|
&\le |a-1/2|+|b-a|\\
&\le\epsilon+\tfrac12
\bigl(\sqrt{V+(h-\epsilon)^2}-(h-\epsilon)\bigr)
=0.0377473...<b_0<h.
\end{aligned}
\end{equation}

For $0\le x\le\ell$, define
\begin{equation}\label{cc:eq:tail}
 q(x)=\frac{V}{V+(c-x)^2},\quad k(x)=H(q(x)),\quad
 \Theta(x)=\int_0^x k(y)\,dy.
\end{equation}
Since $\ell \le c=1/2-\epsilon\le a\le1-c$,
Cantelli's one-sided variance inequality gives,
for $0\le x\le\ell$,
\begin{equation}\label{cc:eq:cantelli}
\begin{aligned}
m_-(x):=\nu\{u<x\}
&\le\frac{V}{V+(a-x)^2}\le q(x),\\
m_+(x):=\nu\{u>1-x\}
&\le\frac{V}{V+(1-x-a)^2}\le q(x),
\end{aligned}
\end{equation}
where $q(x) <1/2$ since $(c-\ell)^2>V$.

Let
\begin{equation}\label{cc:eq:calibration}
 y(x)=\sin(\min\{\Theta(x),\pi/2\}),\quad
 \alpha(x)=\cos(\min\{\Theta(x),\pi/2\}),\quad g(x)=1-y(x).
\end{equation}
These functions are Lipschitz, $\alpha(0)=g(0)=1$, and
\begin{equation}\label{cc:eq:calidentity}
 \alpha k+g'=0\quad\hbox{almost everywhere}.
\end{equation}
The scalar certificate below proves
\begin{equation}\label{cc:eq:localnumbers}
 K(h)>3.316,\quad \Theta(\ell)>1.1533,\quad
 3.316-\frac{1-\sin(1.1533)}{(h-b_0)^2}>\frac1{100}.
\end{equation}
In particular $g(\ell)\le1-\sin(1.1533)$.

\subsubsection{Keeping the endpoint atoms exactly}

Define
\[
A_-(s)=\int_0^{\min\{s,\ell\}}\alpha(x)\,dx,
\qquad
A_+(s)=A_-(1-s).
\]
For $0<u<\ell$, we have
$0\le\alpha(u)\le1$ and
$g(u)=1-\sqrt{1-\alpha(u)^2}$.
By Fenchel duality,
\[
\begin{aligned}
\sqrt{1+|\nabla u|^2}-1
&\ge \alpha(u)|\nabla u|
   +\sqrt{1-\alpha(u)^2}-1\\
&=\alpha(u)|\nabla u|-g(u).
\end{aligned}
\]
By BV locality, $\nabla u=0$ almost everywhere
on $\{u=0,1\}$, so these endpoint atoms incur
no Fenchel penalty.  
Since $A_-$ is $1$-Lipschitz, the BV chain
rule gives
\[
|DA_-(u)|_\nu
\le\int_{\{0<u<\ell\}}\alpha(u)|\nabla u|\,d\nu
+|D^su|_\nu.
\] 
For $0<u<\ell$, Fenchel duality gives
\[
\sqrt{1+|\nabla u|^2}-1
\ge\alpha(u)|\nabla u|-g(u).
\]
The analogous inequality holds on
$\{1-\ell<u<1\}$ with $u$ replaced by $1-u$.
On the core interval $[\ell,1-\ell]$,
we retain the original square-root term.

Since the derivatives of $A_-$, $A_+$, and
the clipping map have disjoint active intervals
and their absolute values sum to at most one,
the BV chain rule ensures that the singular
variation is not counted more than once.
Consequently,
\[
\begin{aligned}
\cG(u)
&\ge \cG(w)+|DA_-(u)|_\nu+|DA_+(u)|_\nu\\
&\quad-\int_{\{0<u<\ell\}}g(u)\,d\nu  -\int_{\{1-\ell<u<1\}}g(1-u)\,d\nu.
\end{aligned}
\]

By the weighted BV coarea formula, Gaussian comparison,
and the monotonicity of $H$,
\[
\begin{aligned}
|DA_-(u)|_\nu
&=\int_0^\ell\alpha(x)
P_\nu(\{u<x\})\,dx\\
&\ge\int_0^\ell\alpha(x)
H(m_-(x))m_-(x)\,dx\\
&\ge\int_0^\ell\alpha(x)k(x)m_-(x)\,dx.
\end{aligned}
\]
Writing $s_-=\nu\{u=0\}$, integration by parts
gives
\[
\begin{aligned}
\int_{\{0<u<\ell\}}g(u)\,d\nu
&=g(\ell)m_-(\ell) -g(0)s_-
-\int_0^\ell g'(x)m_-(x)\,dx.
\end{aligned}
\]
Since $\alpha k+g'=0$ almost everywhere
and $g(0)=1$, the lower-tail contribution
is at least
$
s_--g(\ell)m_-(\ell).
$
The upper-tail estimate follows symmetrically,
with $s_+=\nu\{u=1\}$.
Applying the lower- and upper-tail
estimates yields
\begin{equation}\label{cc:eq:endpointgain}
\cG(u)\ge\cG(w)+s(u)-g(\ell)q_{\rm tail},
\qquad
q_{\rm tail}=m_-(\ell)+m_+(\ell).
\end{equation}

Since $w=1/2-h$ on $\{u<\ell\}$ and
$w=1/2+h$ on $\{u>1-\ell\}$,
\eqref{cc:eq:meanclip} implies
$
|w-b|
\ge h-|b-\tfrac12|
\ge h-b_0
$
on both tails, where
$b=\nu w$. Consequently,
\[
\Var_\nu(w)
\ge(h-b_0)^2
\bigl(m_-(\ell)+m_+(\ell)\bigr)
=(h-b_0)^2q_{\rm tail}.
\] 
Combining it with \eqref{cc:eq:core},
\eqref{cc:eq:localnumbers}, and
\eqref{cc:eq:endpointgain}, we obtain
\[
\begin{aligned}
\cF(u)
&=1-s(u)+\cG(u)\\
&\ge1+
\left(K(h)-\frac{g(\ell)}{(h-b_0)^2}\right)
\Var_\nu(w)\\
&\ge1+\frac1{100}\Var_\nu(w),
\end{aligned}
\]
which proves \eqref{cc:eq:localclaim}.

If $\cF(u)\le1$, then $w$ is constant.
By \eqref{cc:eq:meanclip}, its constant value
lies strictly inside $(\ell,1-\ell)$.
Hence $u=w=a$ almost everywhere, completing
the proof of Proposition~\ref{cc:prop:local}.

\subsection{Gaussian localization of a low-perimeter competitor}
We next prove the rigidity that forces any low-perimeter competitor satisfies the assumptions of Proposition~\ref{cc:prop:local}.

It suffices to consider $n\ge2$ and a finite-perimeter $E\subset Q_n$ with volume $a$ and $P(E;Q_n)\le1$. Set
\[
 A=(\Phi^{\otimes n})^{-1}(E),\quad s=\Phi^{-1}(a),\quad t=|s|,
 \quad p=P_\gamma(A),\quad b_A=\int_Az\,d\gamma_n(z).
\]
Gaussian perimeter throughout this proof is normalized by probability density. Put
\begin{equation}\label{cc:eq:defects}
 \Delta=p-\phi(t),\qquad \beta=\phi(t)-|b_A|,\qquad
 R=\phi_0P(E;Q_n)-p.
\end{equation}
As in \cite[Proposition~3.2 and Eq.~(5.11)]{Glaudo2024},
the coordinatewise Gaussian transport and its tangential Jacobian give
\begin{align}
R&=\int_{\partial^*A}
\left[\sqrt{\sum_i\mathbf n_{A,i}^2e^{z_i^2}}-1\right]d\sigma_A
\ge0,\label{cc:eq:area}\\
0& \le\Delta+R\le\phi_0-\phi(t).\label{cc:eq:budget}
\end{align}
Here $\mathbf n_A$ is the outward unit normal and $\sigma_A$
the Gaussian boundary measure. The same comparison in Glaudo's
Section~5 is dimension-free, but does not yield the explicit
near-median constants needed below.

The improved version of \cite[Eqs.~(30)--(31)]{BarchiesiBrancoliniJulin2017}
gives, after converting to our probability-density normalization,
\begin{equation}\label{cc:eq:}
\beta\le C(t)\Delta,\qquad
C(t)=12\left(1+\frac{e^{-t^2}}{\pi\Phi(-t)^2}\right).
\end{equation} 
This improved coefficient is essential for the
certified estimates in Section~A.5:
it gives $C(t)<C_0=28.17738$, which controls
both the normal excess in the test-field
transfer estimate and the final fiber-variance
bound. The weaker general stability constant
would not suffice for the present numerical
certificate.
The function $C$ is increasing for $t\ge0$,
by the Gaussian Mills-ratio bound
$\phi(t)/\Phi(-t)>t$.
Glaudo \cite[Eqs.~(5.13)--(5.15)]{Glaudo2024}
uses the same stability mechanism, whereas
we retain the explicit coefficient $C(t)$
needed for our quantitative estimates.

The numerical bounds below ensure $b_A\ne0$. Let
\begin{equation}\label{cc:eq:angles}
 v=-\frac{b_A}{|b_A|},\quad H_v=\{z:v\cdot z<s\},\quad
 d=1-\max_i v_i^2,\quad N=2(\Delta+\beta).
\end{equation}
Gaussian Gauss--Green (cf. \cite[Eq.~(5.14)]{Glaudo2024})
gives the exact normal-excess identity
\begin{equation}\label{cc:eq:normalerror}
\int_{\partial^*A}|\mathbf n_A-v|^2\,d\sigma_A=N.
\end{equation}

\begin{lemma}[Column control and affine cancellation]\label{cc:lem:columns}
One has
\begin{align}
 \beta&=\int_{A\symdiff H_v}|v\cdot z-s|\,d\gamma_n(z),\label{cc:eq:betamoment}\\
 \int(\one_A-\one_{H_v})z\,d\gamma_n&=\beta v.\label{cc:eq:alignment}
\end{align}
Writing $z=y+rv$ with $y\in v^\perp$, let $m(y)$ be the one-dimensional Gaussian mass of the symmetric difference in the corresponding column. Then
\begin{equation}\label{cc:eq:column}
 \int m(y)^2\,d\gamma_{v^\perp}(y)\le4\phi_0\beta.
\end{equation}
The signed integral of $\one_A-\one_{H_v}$ against any affine function of $y$ is zero.
\end{lemma}
\begin{proof}
Observe $r=v\cdot z$. Since $H_v=\{r<s\}$, we have
\[
(\one_A-\one_{H_v})(r-s)
=\one_{A\symdiff H_v}|r-s|.
\] 
Integrating both sides and noting that $A$ and $H_v$ have equal volumes
and $b(H_v)=-\phi(s)v$, we obtain  
\eqref{cc:eq:betamoment}--\eqref{cc:eq:alignment}.
 
For each fixed $y\in v^\perp$, let $B_+\subset\{ r\ge s\}$ and $B_-\subset\{ r\le s\}$
denote the symmetric-difference sections on the
two sides of $r=s$, with Gaussian masses
$m_+$ and $m_-$, respectively.
Since $\phi\le\phi_0$, for any $t\ge0$,
\[
\int_{B_+\cap(s+t,\infty)}\phi(r)\,dr
\ge(m_+-\phi_0t)_+.
\]
Hence the layer-cake formula gives
\[
\begin{aligned}
\int_{B_+}(r-s)\phi(r)\,dr
&=\int_0^\infty
\left(\int_{B_+\cap(s+t,\infty)}
\phi(r)\,dr\right)dt\\
&\ge\int_0^\infty(m_+-\phi_0t)_+\,dt =\frac{m_+^2}{2\phi_0}.
\end{aligned}
\]
The same argument applies to $B_-$. Consequently,
\[
\begin{aligned}
\int_{A_y\symdiff(H_v)_y}|r-s|\phi(r)\,dr
&\ge\frac{m_+^2+m_-^2}{2\phi_0} \ge\frac{m(y)^2}{4\phi_0},
\end{aligned}
\]
where $m(y)=m_++m_-$. Integrating over
$y\in v^\perp$ and using
\eqref{cc:eq:betamoment} proves
\eqref{cc:eq:column}.
Finally, projecting \eqref{cc:eq:alignment}
onto $v^\perp$ and using equality of volumes
shows that
\[
\int(\one_A-\one_{H_v})L(y)\,d\gamma_n=0
\]
for every affine function $L$ on $v^\perp$.
\end{proof}

\subsection{Quantitative coordinate localization via a bounded test field}

We refine the test-field-transfer argument of
Lemma~\ref{lem:masterrigid} for the uniform cube.
The general estimate is sufficient for qualitative
coordinate localization, but too coarse for the
explicit interval considered here. The main
improvement is to exploit affine cancellation
and Gaussian column estimates, rather than
bounding the divergence error pointwise on the
symmetric difference. This yields the sharper
quantitative estimate below.

For a cutoff $r>0$, define the even $C^{1,1}$ function
\begin{equation}\label{cc:eq:q}
 q_r(z)=
 \begin{cases}
 z^2/2,&|z|\le r,\\
 -z^2/2+2r|z|-r^2,&r\le|z|\le2r,\\
 r^2,&|z|\ge2r.
 \end{cases}
\end{equation}
It obeys
\begin{equation}\label{cc:eq:qprops}
 0\le q_r\le\min(z^2/2,r^2),\quad |q_r'|\le r,\quad
 |q_r''|\le1\ \hbox{a.e.},\quad |q_r'(z)-z|\le2(|z|-r)_+.
\end{equation}
For $v_0>0$ and $h_0>0$, set
\begin{align}
 M(v_0,h_0)&=\E[(|\sqrt{v_0}G|-h_0)_+^2] \notag\\
 &=2\bigl[(v_0+h_0^2)\Phi(-h_0/\sqrt{v_0})
       -h_0\sqrt{v_0}\phi(h_0/\sqrt{v_0})\bigr],\label{cc:eq:M}\\
 M_*(v_0,r,m_0)&=\E[(|m_0+\sqrt{v_0}G|-r)_+^2]\notag\\
 &=\tfrac12\{M(v_0,r-m_0)+M(v_0,r+m_0)\}\qquad(0\le m_0<r),\label{cc:eq:Mnoncentral}\\
 Q_r(v_0)&=\E q_r(\sqrt{v_0}G)
 =\frac{v_0}{2}-M(v_0,r)+\tfrac12M(v_0,2r).\label{cc:eq:Q}
\end{align}
Here $G$ is standard normal. The noncentral moment $M_*$ is increasing both in variance and in $m_0\ge0$: its integrand is convex, even, and nondecreasing in absolute value, so convolution with a centered Gaussian preserves these properties, and increasing variance is another centered Gaussian convolution.

The function $Q_r$ is increasing and concave; differentiation gives
\[
 Q_r''(v_0)=\frac r{v_0^{3/2}}
 \bigl[\phi(2r/\sqrt{v_0})-\phi(r/\sqrt{v_0})\bigr]<0.
\]
Put $\bar{d} =1-d$, $m_d=\max(d,\bar{d})$, and
\begin{equation}\label{cc:eq:minor}
 B_r(d)=\left(\frac{4r^3\phi_0}{3\sqrt{m_d}}\right)^{1/2}.
\end{equation}
Define
\begin{align}
 \mathcal R_r(d,t)={}&2\bar{d}^{3/2}\sqrt{M_*(d,r,t\sqrt{\bar{d}})}
 +B_r(d)\,d\sqrt{\min(d,\bar{d})} +t\left(\frac{\bar{d}d}{\sqrt2}+\frac{r^2d}{2}\right),\label{cc:eq:rem}\\
 L_r(d)={}&
 \begin{cases}
 (1-d)Q_r(d)+dQ_r(1-d),&0\le d\le1/2,\\
 Q_r(d),&1/2\le d\le1.
 \end{cases}\label{cc:eq:lower}
\end{align}
The first part of $L_r$ is increasing: $f(w)=wQ_r(1-w)$ is concave, as $f''=-2Q_r'+wQ_r''\le0$, and $L_r(d)=f(d)+f(1-d)$. The second part is increasing as well, and the pieces agree at $1/2$.
Let 
\[
Q_v(z)=\sum_i v_i^2q_r(z_i), \qquad 
\mathcal Q_H=\int_{\partial H_v}Q_v\,d\sigma_{H_v}.
\]

\begin{lemma}[Test-field transfer]\label{cc:lem:transfer}
For $0<d<1$ and $t<r$,
\begin{equation}\label{cc:eq:transfer}
 \phi(t)L_r(d)\le \mathcal Q_H
\le
 (\sqrt R+r\sqrt N)^2+(1+r^2+tr)\beta
 +2\sqrt{\phi_0\beta}\,\mathcal R_r(d,t).
\end{equation}
\end{lemma}
\begin{proof}
Put
$
W(z,\mathbf n_A)
:=\sqrt{\sum_i\mathbf n_{A,i}^2e^{z_i^2}}.
$
Since $\sum_i\mathbf n_{A,i}^2=1$,
the weighted RMS--AM inequality gives
\[
W-1
\ge\sum_i\mathbf n_{A,i}^2
(e^{z_i^2/2}-1)
\ge\sum_i\mathbf n_{A,i}^2q_r(z_i).
\]
Integrating over $\partial^*A$ and using
\eqref{cc:eq:area} yields
\[
\int_{\partial^*A}
\sum_i\mathbf n_{A,i}^2q_r(z_i)\,d\sigma_A
\le R.
\]
Recall that
$N=2(\Delta+\beta)
=\int_{\partial^*A}|\mathbf n_A-v|^2\,d\sigma_A$.
The $L^2$ triangle inequality and \eqref{cc:eq:normalerror} yield
\begin{equation}\label{cc:eq:boundaryL2}
 \left(\int_{\partial^*A}Q_v\,d\sigma_A\right)^{1/2}
 \le\sqrt R+r\sqrt N.
\end{equation}

For $X=vQ_v$, write
\[
D(z):=\divg X(z)
=J_v(z)-(v\cdot z)Q_v(z),
\qquad
J_v(z)=\sum_i v_i^3q_r'(z_i).
\]
On the hyperplane $v\cdot z=s$, define
\[
D_s(y):=D(y+sv)
=J_v(y+sv)-sQ_v(y+sv),
\qquad y\in v^\perp.
\] 
For each fixed $y\in v^\perp$, write
$z=y+\xi v$, where $\xi\in\mathbb R$.
Differentiating with respect to $\xi$ gives
\[
\partial_\xi J_v
=\sum_i v_i^4q_r''(z_i),
\qquad
\partial_\xi Q_v
=\sum_i v_i^3q_r'(z_i)=J_v.
\]
Since $|v|=1$ and
$0\le q_r\le r^2$, $|q_r'|\le r$,
$|q_r''|\le1$ a.e., we obtain
\[
|\partial_\xi J_v|
\le\sum_i v_i^4\le1,\qquad
Q_v\le r^2,\qquad
|\partial_\xi Q_v|
\le r\sum_i|v_i|^3\le r.
\]
Hence,
\[
\begin{aligned}
&|D(y+\xi v)-D(y+sv)|\\
&\le |J_v(y+\xi v)-J_v(y+sv)| +|\xi Q_v(y+\xi v)-sQ_v(y+sv)|\\
&\le |\xi-s|+r^2|\xi-s|
   +|s|r|\xi-s|.
\end{aligned}
\]
Here the last step follows by writing
\[
\begin{aligned}
&\xi Q_v(y+\xi v)-sQ_v(y+sv)\\
&=(\xi-s)Q_v(y+\xi v) +s\bigl(Q_v(y+\xi v)-Q_v(y+sv)\bigr)
\end{aligned}
\]
and using $Q_v\le r^2$ and
$|\partial_\xi Q_v|\le r$.
Since $t=|s|$, we conclude that
\begin{equation}\label{cc:eq:displacement}
|D(y+\xi v)-D(y+sv)|
\le(1+r^2+tr)|\xi-s|.
\end{equation}

Choose $j$ with $v_j^2=\bar d$.
Conditionally on $v\cdot z=s$, we have
$z_j\sim N(sv_j,d)$. Since $z_j=y_j+sv_j$
is affine in $y\in v^\perp$, we may subtract
$v_j^3z_j$ from the main-coordinate term
$v_j^3q_r'(z_j)$ in $J_v$. By
\eqref{cc:eq:qprops},
\[
\begin{aligned}
\|v_j^3(q_r'(z_j)-z_j)\|_{L^2} 
&\le 2|v_j|^3
\bigl(\E(|z_j|-r)_+^2\bigr)^{1/2} =2\bar d^{3/2}
\sqrt{M_*(d,r,t\sqrt{\bar d})},
\end{aligned}
\]
where $t=|s|$. 
This proves the first term in \eqref{cc:eq:rem}. 

For $i\ne j$, $v_i^2\le\min(d,\bar{d})$, so the conditional variance of $z_i$ is at least $m_d$. Regardless of its mean, its density is at most $\phi_0/\sqrt{m_d}$. Since
$
 \int_{\R}q_r'(z)^2\,dz=4r^3 / 3,
$
we have $\|q_r'(z_i)\|_2\le B_r(d)$. Also
$\sum_{i\ne j}|v_i|^3\le d\sqrt{\min(d,\bar{d})}$. This proves the second term in \eqref{cc:eq:rem}.

For the $Q_v$ term, we use the second-order
Gaussian Poincar\'e inequality
\[
\inf_{\ell_0\ \mathrm{affine}}
\|q_r(m+\sigma G)-\ell_0(G)\|_2
\le\frac{\sigma^2}{\sqrt2},
\qquad \|q''\|_\infty\le1.
\]
Indeed, expanding $f(G)=q_r(m+\sigma G)$
in the orthonormal Hermite basis (with coefficients $a_k$) gives
\[
\begin{aligned}
\inf_{\ell_0\ \mathrm{affine}}
\|f-\ell_0\|_2^2
&=\sum_{k\ge2}a_k^2 \le\frac12\sum_{k\ge2}k(k-1)a_k^2 =\frac12\|f''\|_2^2
\le\frac{\sigma^4}{2}.
\end{aligned}
\]
Approximation extends this estimate to
$C^{1,1}$ functions.

Conditionally on $v\cdot z=s$, we have
$z_j=sv_j+\sqrt d\,G$, where
$G=y_j/\sqrt d$ is linear in $y\in v^\perp$.
Thus there exists an affine function $\ell_j(y)$ (explicitly, $\ell_j(y)=\bar d\left[
\E q_r(Z_j)+y_j\E q_r'(Z_j)
\right]$)
such that
\[
\|\bar d\,q_r(z_j)-\ell_j(y)\|_2
\le\frac{\bar d d}{\sqrt2}.
\]
For $i\ne j$, let
$c_i=\E q_r(z_i)$, where expectations are
taken under the conditional Gaussian measure
on $v\cdot z=s$.
Since $0\le q_r\le r^2$, the bounded-variable
variance inequality gives
\[
\|q_r(z_i)-c_i\|_2\le\frac{r^2}{2}.
\]
Consequently, the $L^2$ triangle inequality
yields
\[
\begin{aligned}
&\left\|\sum_{i\ne j}v_i^2
(q_r(z_i)-c_i)\right\|_2  \le\frac{r^2}{2}
\sum_{i\ne j}v_i^2
=\frac{r^2d}{2}.
\end{aligned}
\] 
It follows that the affine function
$
\ell_Q(y)=\ell_j(y)+\sum_{i\ne j}v_i^2c_i
$
satisfies
\[
\|Q_v(y+sv)-\ell_Q(y)\|_2
\le\frac{\bar d d}{\sqrt2}+\frac{r^2d}{2}.
\]

The preceding estimates for $J_v$ similarly
give an affine function
$\ell_J(y)=v_j^3(y_j+sv_j)$ satisfying
\[
\begin{aligned}
\|J_v(y+sv)-\ell_J(y)\|_2
&\le 2\bar d^{3/2}
\sqrt{M_*(d,r,t\sqrt{\bar d})} +B_r(d)d\sqrt{\min(d,\bar d)}.
\end{aligned}
\]
Finally, since
$D_s(y)=J_v(y+sv)-sQ_v(y+sv)$
and $t=|s|$, the affine function
$\ell=\ell_J-s\ell_Q$ satisfies
\[
\begin{aligned}
\|D_s-\ell\|_2
&\le\|J_v-\ell_J\|_2
+t\|Q_v-\ell_Q\|_2 \le\mathcal R_r(d,t).
\end{aligned}
\]
Hence
\[
\inf_{\ell\ \mathrm{affine}}
\|D_s-\ell\|_{L^2(\gamma_{v^\perp})}
\le\mathcal R_r(d,t).
\]

To transfer this bound from $H_v$ to $A$,
we first cancel the affine part using
Lemma~\ref{cc:lem:columns}. Indeed, if
\[
m(y)=\int_{\mathbb R}
|\one_A-\one_{H_v}|(y+\xi v)
\phi(\xi)\,d\xi,
\]
then $\|m\|_2^2\le4\phi_0\beta$,
and the integral of
$(\one_A-\one_{H_v})\ell(y)$ vanishes
for every affine $\ell$. Hence
Cauchy--Schwarz gives
\begin{align} \label{cc:eq:integralD}
\left|\int(\one_A-\one_{H_v})D_s(y)
\,d\gamma_n\right|
\le2\sqrt{\phi_0\beta}\,
\mathcal R_r(d,t).
\end{align}
Moreover, \eqref{cc:eq:displacement} and
\eqref{cc:eq:betamoment} imply
\begin{align} \label{cc:eq:integralDD}
\left|\int(\one_A-\one_{H_v})
\bigl(D (z)-D_s(y)\bigr)
\,d\gamma_n\right|
\le(1+r^2+tr)\beta.
\end{align}

By Gaussian Gauss--Green and the fact that
$v$ is the outward normal to $H_v$,
\[
\begin{aligned}
\mathcal Q_H =\int_{H_v}D\,d\gamma_n 
&=\int_{\partial^*A}
Q_v(\mathbf n_A\cdot v)\,d\sigma_A
+\int(\one_{H_v}-\one_A)D\,d\gamma_n\\
&\le\int_{\partial^*A}Q_v\,d\sigma_A
+\int(\one_{H_v}-\one_A)D\,d\gamma_n,
\end{aligned}
\]
where we used $Q_v\ge0$ and
$\mathbf n_A\cdot v\le1$.  
Substituting the estimates in  \eqref{cc:eq:boundaryL2},  \eqref{cc:eq:integralD},  and  \eqref{cc:eq:integralDD}  into it yields
\[
\begin{aligned}
\mathcal Q_H
&\le(\sqrt R+r\sqrt N)^2  +(1+r^2+tr)\beta +2\sqrt{\phi_0\beta}\,
\mathcal R_r(d,t).
\end{aligned}
\]
The bounded Lipschitz field $X$ has
at-most-linear Gaussian divergence,
so Gauss--Green is justified by cutoffs
and smooth approximation.

Finally, conditionally on $v\cdot z=s$,
each coordinate $z_i$ has mean $sv_i$
and variance $1-v_i^2$. Since $q_r$ is
even and nondecreasing in $|z|$,
\[
\begin{aligned}
\mathcal Q_H
&=\phi(t)\sum_i v_i^2
\E q_r(sv_i+\sqrt{1-v_i^2}\,G)\\
&\ge\phi(t)\sum_i v_i^2Q_r(1-v_i^2).
\end{aligned}
\]
If $d\le1/2$, the main coordinate has
conditional variance $d$, whereas each
remaining coordinate has conditional
variance at least $1-d$. Therefore
\[
\mathcal Q_H\ge\phi(t)
\bigl[(1-d)Q_r(d)+dQ_r(1-d)\bigr].
\]
If $d\ge1/2$, every conditional variance
is at least $d$, giving
$\mathcal Q_H\ge\phi(t)Q_r(d)$.
In both cases,
\[
\mathcal Q_H\ge\phi(t)L_r(d),
\]
which completes the proof of
\eqref{cc:eq:transfer}.

\end{proof}

\subsection{Certified numerical closure}
All decimal constants in this section are exact rationals. Take
\begin{equation}\label{cc:eq:globalconstants}
 \begin{gathered}
 \epsilon=0.0144,\quad T=0.03610329,\quad
 \phi_-=0.39868236,\quad \phi_+=0.39894229,\\
 C_0=28.17738,\quad D_0=0.0002599155,\quad k_0=0.108107,\\
 r=0.6875=11/16,\quad \Lambda=70.176,\quad d_0=0.0849.
 \end{gathered}
\end{equation}
The certificate verifies
\begin{equation}\label{cc:eq:scalarbounds}
 t<T,\quad \phi_-<\phi(t)\le\phi_0<\phi_+,\quad
 C(t)<C_0,\quad \Delta+R<D_0,\quad
 2\sqrt{\phi_0\beta}<k_0.
\end{equation}
For example, the first and fourth bounds follow from
$\Phi(T)-1/2>\epsilon$ and $\phi_0(1-e^{-T^2/2})<D_0$, respectively. The  bound is checked at $T$ using monotonicity of $C$; $k_0^2>4\phi_+C_0D_0$. Also $C_0D_0<\phi_-$, ensuring $b_A\ne0$.

Let $B_0=2(1+C_0)r^2$ and $B=B_0+(1+r^2+Tr)C_0$. The exact comparisons
\begin{equation}\label{cc:eq:matrix}
 \Lambda>1,\quad \Lambda>B,\quad
 (\Lambda-1)(\Lambda-B)>B_0
\end{equation}
show that
\[
 (\sqrt R+r\sqrt N)^2+(1+r^2+tr)\beta
 \le\Lambda(R+\Delta)<\Lambda D_0.
\]
Combining Lemma~\ref{cc:lem:transfer} with
\eqref{cc:eq:scalarbounds} and
\eqref{cc:eq:matrix}, we conclude that
every feasible $d>0$ must satisfy
\begin{equation}\label{cc:eq:necessary}
\phi_-L_r(d)
\le \Lambda D_0+k_0\mathcal R_r(d,T).
\end{equation}

\subsubsection{Exclusion on a complete interval, not a grid of points}
The certificate divides $[d_0,1]$ into $9151$ closed intervals of length $10^{-4}$, with $1/2$ a grid boundary. On each cell $[d_l,d_u]$, monotonicity gives the left-hand lower bound $\phi_-L_r(d_l)$. Put
\[
 w_u=1-d_l,\qquad
 m_l=\begin{cases}1-d_u,&d_u\le1/2,\\d_l,&d_l\ge1/2.\end{cases}
\]
An upper bound for the remainder throughout the cell is
\begin{align}\label{cc:eq:cellupper}
 \widehat{\mathcal R}={}&
 2w_u^{3/2}\sqrt{M_*(d_u,r,T\sqrt{w_u})}
 +\left(\frac{4r^3\phi_+}{3\sqrt{m_l}}\right)^{1/2}
       d_u\sqrt{\min(d_u,w_u)} +T\left(\frac{w_ud_u}{\sqrt2}+\frac{r^2d_u}{2}\right).
\end{align}
Monotonicity of $M_*$ in variance and absolute mean proves the first term's bound; all other factors are bounded individually in the favorable direction. The certificate checks
\begin{equation}\label{cc:eq:celltest}
 \phi_-L_r(d_l)>\Lambda D_0+k_0\widehat{\mathcal R}
\end{equation}
for every cell. The smallest certified margin is larger than $2.09\times10^{-5}$. Consequently
\begin{equation}\label{cc:eq:anglebound}
 d<0.0849=d_0.
\end{equation}

\subsubsection{From a coordinate barycenter to fiber variance} 
Choose $j$ with $v_j^2=1-d$, and let $u$
denote the mass of the corresponding cube fiber.
One-dimensional Gaussian rearrangement gives
$|(b_A)_j|\le\E\IG(u)$.
Since
$\IG''(s)=-1/\IG(s)\le-1/\phi_0$
and $\E u=a$, strong concavity yields
\[
\E\IG(u)\le\IG(a)
-\frac{\Var(u)}{2\phi_0}.
\]
Consequently,
\begin{equation}\label{cc:eq:fibervariance}
\Var(u)\le
2\phi_0\bigl(\phi(t)-|(b_A)_j|\bigr).
\end{equation}
Since $|b_A|=\phi(t)-\beta$ and
$|v_j|=\sqrt{1-d}$, we have
\[
\phi(t)-|(b_A)_j|
=\beta+(\phi(t)-\beta)(1-\sqrt{1-d}).
\]
The right-hand side is nondecreasing in
$\beta$, $\phi(t)$, and $d$, since
$0\le\beta\le\phi(t)$ and $0\le d<1$.
By the quantitative Gaussian barycentric stability
estimate \eqref{cc:eq:} and
\eqref{cc:eq:scalarbounds},
$
\beta\le C(t)\Delta<C_0D_0.
$
Together with \eqref{cc:eq:anglebound},
this yields
\begin{align}\label{cc:eq:varianceclosure}
\Var(u)
&<2\phi_+\left[
C_0D_0+(\phi_+-C_0D_0)
(1-\sqrt{1-d_0})\right]\notag\\
&<0.019403<0.01942=V.
\end{align}
Thus the hypotheses of
Proposition~\ref{cc:prop:local} are satisfied.
 
\subsection{Completion and equality cases}

By \eqref{cc:eq:varianceclosure}, the fiber-volume
function $u$ satisfies the hypotheses of
Proposition~\ref{cc:prop:local}. Thus, for any
cube competitor with $P(E;Q_n)\le1$,
\[
1\ge P(E;Q_n)\ge\cF(u)
\ge1+\frac1{100}
\Var(\clip(u,[0.301,0.699]))\ge1.
\]
The equality statement of
Proposition~\ref{cc:prop:local} yields
$u\equiv a$. As in the proof of
Theorem~\ref{thm:general}, the BV slicing
formula then gives
\[
1\ge P(E;Q_n)
\ge\int_{\partial^*E}
|\mathbf n_{E,j}|\,d\mathcal H^{n-1}
=\int p(y)\,dy\ge1.
\]
Equality forces $|\mathbf n_{E,j}|=1$
almost everywhere on the reduced boundary,
so all transverse distributional derivatives
of $\one_E$ vanish. Hence $E$ is a
coordinate cylinder over a one-dimensional
relative isoperimetric minimizer, namely
an endpoint-attached interval. The case
$n=1$ is immediate.

For products of circles, the profile identity
\eqref{cc:eq:profileidentity} gives the
corresponding perimeter equality.
The equality cases follow by the same
fiber-rearrangement and BV slicing argument
as in the proof of Theorem~\ref{thm:manifold}:
the rearranged minimizer is a coordinate slab,
and equality at the active factor forces
the transverse normal component to vanish.
Thus the original minimizer is a cylinder
over an arc of $C_2$.  

\subsection{Arithmetic certificate}
The accompanying Python certificate
\texttt{certify\_circle\_cube\_0144.py} uses exact rational
arithmetic and outward rounding to $10^{-35}$, with rigorous
Taylor bounds for Gaussian functions and trigonometric quantities.
Its $2048$-term right Riemann sum gives
$\Theta(\ell)>1.15338086892085$, while
\eqref{cc:eq:celltest} checks \emph{entire cells}, not isolated grid
points. The half-width $9/625$ is not claimed optimal.

\section*{Acknowledgements}
\emph{Generative-AI use disclosure.}   
During the preparation of this work, the author used ChatGPT (GPT-6, OpenAI) to explore proof strategies, develop and refine mathematical arguments, and improve the exposition and organization of the manuscript.

The author's broader research goal is to determine the exact dimension-free isoperimetric profiles of general product spaces and characterize the structure of asymptotically minimizing sets.
With ChatGPT's assistance, the author initially established a dimension-free isoperimetric result for cubes on an interval with half-width $6.25\times10^{-12}$. The author then recognized that the underlying approach could be generalized first to weighted probability measures on the real line and subsequently to weighted Riemannian manifolds. ChatGPT assisted with the technical development of these extensions and with further improving the explicit half-width for cubes to $9/625$, as established in the present paper.

The author has independently verified the mathematical arguments and references and takes full responsibility for the content and correctness of the manuscript. 

\noindent \emph{Funding.} This work was supported by the National Key Research and
Development Program of China under grant 2023YFA1009604 and by the
NSFC under grant 62101286.

\bibliographystyle{abbrv} 
\bibliography{ref}
\end{document}